\documentclass[a4paper,11pt]{article}

\usepackage[english]{babel}
\usepackage{lmodern}

\usepackage{tikz}
\usetikzlibrary{calc}
\usepackage{authblk}
\usepackage{booktabs}

\usepackage{caption}		
\usepackage{footnote}		
\makesavenoteenv{tabular}	
\makesavenoteenv{table}		

\usepackage{graphicx}

\usepackage{amsmath,amsfonts,amsthm,authblk,MnSymbol}
\usepackage{mathrsfs}
\usepackage[inline, shortlabels]{enumitem}

\usepackage[mathlines]{lineno}

\usepackage{verbatim}

\usepackage{hyperref}
\usepackage{cleveref}

\Crefname{observation}{Observation}{Observations}
\Crefname{hope}{Hope}{Hopes}
\Crefname{lemma}{Lemma}{Lemmata}
\Crefname{theorems}{Theorems}{Theorems}
\Crefname{counterexample}{Counterexample}{Counterexamples}
\Crefname{conj}{Conjecture}{Conjectures}
\Crefname{prop}{Proposition}{Propositions}
\Crefname{cor}{Corollary}{Corollaries}
\Crefname{example}{Example}{Examples}

\theoremstyle{plain}
\newtheorem{theorem}{{\bf Theorem}}
\newtheorem{prop}[theorem]{Proposition}
\newtheorem{cor}[theorem]{Corollary}

\newtheorem{lemma}[theorem]{Lemma}

\theoremstyle{definition}
\newtheorem{definition}[theorem]{Definition}

\newtheorem{remark}[theorem]{Remark}
\newtheorem{example}[theorem]{Example}

\newcommand{\tetrahedron}{\Delta}

\newcommand{\dual}{\Gamma}

\newcommand{\tri}{\mathcal{T}}
\newcommand{\hbody}{\mathcal{H}}
\newcommand{\spine}{S}

\newcommand{\MCG}{\operatorname{MCG}}
\newcommand{\gens}{X}

\newcommand{\manifold}{\mathcal{M}}

\newcommand{\surface}{\mathcal{S}}
\newcommand{\nsurface}{\mathcal{N}}
\newcommand{\altsurface}{\mathcal{F}}

\newcommand{\alttree}{T}

\newcommand{\R}{\mathbb{R}}

\newcommand{\width}{w}

\newcommand{\tw}[1]{\operatorname{tw} (#1)}
\newcommand{\pw}[1]{\operatorname{pw} (#1)}
\newcommand{\cw}[1]{\operatorname{cw} (#1)}

\newcommand{\sfs}{\operatorname{SFS}}
\newcommand{\lst}{\operatorname{LST}}

\newcommand{\disk}{\mathbb{A}}

\newcommand{\moebius}{\mathbb{M}}
\newcommand{\nsphere}[1]{\mathbb{S}^{#1}}

\title{3-Manifold triangulations with small treewidth}
\author[$\dagger$]{Krist\'of Husz\'ar}
\author[$\ddagger$]{Jonathan Spreer
\thanks{The author is supported by grant EVF-2015-230 of the Einstein Foundation Berlin as well as by the DFG Collaborative Research Center SFB/TRR 109 ``Discretization in Geometry and Dynamics''.}}

\affil[$\dagger$]{{\em IST Austria, Klosterneuburg, Austria}, \href{mailto:kristof.huszar@ist.ac.at}{\nolinkurl{kristof.huszar@ist.ac.at}}}
\affil[$\ddagger$]{{\em Institut f\"ur Mathematik, Freie Universit\"at Berlin, Germany}, \href{mailto:jonathan.spreer@fu-berlin.de}{\nolinkurl{jonathan.spreer@fu-berlin.de}}}

\date{}

\begin{document}

\maketitle

\begin{abstract}
Motivated by fixed-parameter tractable (FPT) problems in computational topology, we consider the treewidth of a compact, connected 3-manifold $\manifold$ defined by
\[
	\tw{\manifold} = \min\{\tw{\dual(\tri)}:\tri~\text{is~a~triangulation~of}~\manifold\},
\]
where $\dual(\tri)$ denotes the dual graph of $\tri$.
In this setting the relationship between the topology of a 3-manifold and its treewidth is of particular interest.

First, as a corollary of work of Jaco and Rubinstein, we prove that for any closed, orientable 3-manifold $\manifold$ the treewidth $\tw{\manifold}$ is at most $4\mathfrak{g}(\manifold)-2$ where $\mathfrak{g}(\manifold)$ denotes the Heegaard genus of $\manifold$. In combination with our earlier work with Wagner, this yields that for non-Haken manifolds the Heegaard genus and the treewidth are within a constant factor.

Second, we characterize all 3-manifolds of treewidth one: These are precisely the lens spaces and a single other Seifert fibered space. Furthermore, we show that all remaining orientable Seifert fibered spaces over the 2-sphere or a non-orientable surface have treewidth two. In particular, for every spherical 3-manifold we exhibit a triangulation of treewidth at most two.

Our results further validate the parameter of treewidth (and other related parameters such as cutwidth, or congestion) to be useful for topological computing, and also shed more light on the scope of existing FPT algorithms in the field.
\end{abstract}

\noindent 
{\em 2010 Mathematics Subject Classification}: 
{\bf 57Q15};	
57N10,		
05C75,		
57M15.		

\medskip

\noindent 
{\em 2012 ACM Subject Classification}:
Mathematics of computing $\rightarrow$ Geometric topology,
Theory of computation $\rightarrow$ Fixed parameter tractability. 

\medskip

\noindent 
{\em Keywords and phrases}: computational 3-manifold topology, fixed-parameter tractability, layered triangulations, structural graph theory, treewidth, cutwidth, Heegaard genus, lens spaces, Seifert fibered spaces.

\section{Introduction}
\label{sec:intro}

Any given topological $3$-manifold $\manifold$ admits infinitely many combinatorially distinct triangulations $\tri$, and the feasibility of a particular algorithmic task about $\manifold$ might greatly depend on the choice of the input triangulation $\tri$. Hence, it is an important question in computational topology, how ``well-behaved'' a triangulation can be, taking into account ``topological properties'' of the underlying 3-manifold.

More concretely, there exist several algorithms in computational $3$-manifold topology which solve inherently difficult (e.g., {\textbf{NP}-hard) problems in linear time in the input size, once the input triangulation has a dual graph of bounded treewidth \cite{burton2017courcelle, burton2016parameterized, Burton2018, burton2013complexity, Maria17Beta1}. Such fixed-parameter tractable (FPT) algorithms are not only of theoretical importance but also provide practical tools: some of them are implemented in software packages such as {\em Regina} \cite{burton2013regina, Regina} and, in selected cases, outperform previous state-of-the-art methods.

The presence of algorithms FPT in the treewidth of the dual graph of a triangulation immediately poses the following question. Given a $3$-manifold $\manifold$, how small can the treewidth of the dual graph of a triangulation of $\manifold$ be? This question has recently been investigated in a number of contexts, settling, for instance, that for some $3$-manifolds there is no hope of finding triangulations with dual graphs of small treewidth \cite{HSWTreewidth_Arxiv} (see \cite{tw-knots} for related work concerning the respective question about knots and their diagrams). Hyperbolic 3-manifolds nevertheless always admit triangulations of treewidth upper-bounded by their volume \cite{tw-hyperbolic}. 

In this article we also focus on constructing small treewidth triangulations informed by the topological structure of a $3$-manifold. To this end, we consider the notion of treewidth (cutwidth) of a $3$-manifold as being the smallest treewidth (cutwidth) of a dual graph ranging over all triangulations thereof. The necessary background is introduced in \Cref{sec:prelims}.

\medskip

In \Cref{sec:tw3mfd}, building on \cite{Jaco06Layering}, we show that the Heegaard genus dominates the cutwidth (and thus the treewidth as well) by virtue of the following statement.

\begin{theorem}
\label{thm:cutwidth}
Let $\manifold$ be a closed, orientable $3$-manifold, and let $\cw{\manifold}$ and $\mathfrak{g}(\manifold)$ respectively denote the cutwidth and the Heegaard genus of $\manifold$. We have $\cw{\manifold} \leq 4\mathfrak{g}(\manifold) - 2$.
\end{theorem}

\Cref{thm:cutwidth}, in combination with recent work by the authors and Wagner \cite{HSWTreewidth_Arxiv}, implies that for the class of so-called non-Haken $3$-manifolds, the Heegaard genus is in fact within a constant factor of both the cutwidth and the treewidth of a $3$-manifold, providing an interesting connection between a classical topological invariant and topological properties directly related to the triangulations of a manifold. 

In \Cref{sec:tw_one}, we further strengthen this link by looking at very small values of Heegaard genus and treewidth:

\begin{theorem}
\label{thm:main}
The class of $3$-manifolds of treewidth at most one coincides with that of Heegaard genus at most one together with the Seifert fibered space $\sfs[\nsphere{2} : (2,1),(2,1),(2,-1)]$.
\end{theorem}

Additionally, in \Cref{sec:tw_two}, we construct treewidth two triangulations for all orientable Seifert fibered spaces over $\nsphere{2}$ (\Cref{thm:sfs1}) or a non-orientable surface (\Cref{thm:sfs2}), which, together with \Cref{thm:main}, yield a number of corollaries. In particular, the treewidth of all $3$-manifolds with spherical or $\nsphere{2} \times \mathbb{R}$ geometry equals to two (\Cref{cor:lowtwhighhg}). 

Finally, combining these results, we determine the treewidth of $4889$ out of the $4979$ manifolds in the $(\leq 10)$-tetrahedra census (\Cref{tab:tw}). Specifically, only $90$ of them have treewidth possibly higher than two. These computations also confirm that not all minimal triangulations are of minimum treewidth (\Cref{cor:minimal}).

\begin{remark} 
  The various triangulations described in \Cref{sec:tw_two} are available in form of a short {\em Regina} script \cite{burton2013regina, Regina} in \Cref{app:python}.
\end{remark}

\paragraph*{Acknowledgements.} The first author thanks the Einstein Foundation (grant EVF-2015-230) for partially supporting his visits at Freie Universit\"at Berlin, and for the hospitality of the people at the Discrete Geometry Group.\\
\hspace*{\parindent}We thank the developers of the free software {\em Regina} \cite{burton2013regina, Regina} for creating a fantastic tool.

\section{Preliminaries}
\label{sec:prelims}

\subsection{Graphs}
\label{ssec:graphs}

A {\em graph} (more precisely, a {\em multigraph}) $G=(V,E)$ is an ordered pair consisting of a finite set $V=V(G)$ of {\em nodes} and of a multiset $E=E(G)$ of unordered pairs of nodes, called {\em arcs}.\footnote{Throughout the article, the terms {\em edge} and {\em vertex} denote an edge or vertex of a triangulated surface or 3-manifold, while the words {\em arc} and {\em node} refer to an edge or vertex in a graph.} A {\em loop} is an arc $e \in E$ which is a multiset itself, e.g., $e = \{v,v\}$ for some $v \in V$.
The {\em degree} $\deg(v)$ of a node $v \in V$ equals the number of arcs containing it, counted with multiplicity.
If all of its nodes have the same degree $k \in \mathbb{N}$, a graph is called {\em $k$-regular}. A {\em tree} is a connected graph with $n$ nodes and $n-1$ arcs.
The term {\em leaf} denotes a node of degree one.

For general background on graph theory we refer to \cite{diestel2017graph}.

\paragraph*{Treewidth.} Originating from graph minor theory \cite{lovasz2006graph} and central to parametrized complexity \cite[Part III]{downey13-fundamentals}, treewidth \cite{bodlaender1994tourist,bodlaender2005discovering, robertson1986graph} measures the similarity of a given graph to a tree.
More precisely, a {\em tree decomposition} of $G=(V,E)$ is a pair $(\{B_i:i \in I\},\alttree=(I,F))$ with {\em bag}s $B_i \subseteq V$, $i \in I$, and a tree $\alttree=(I,F)$, such that
\begin{enumerate*}[a)]
	\item \label{twpropone} $\bigcup_{i \in I} B_i = V$,
	\item \label{twproptwo} for every arc $\{u,v\} \in E$, there exists $i \in I$ with $\{u,v\} \subseteq B_i$, and
	\item \label{twpropthree} for every $v \in V$, $T_v = \{i \in I:v \in B_i\}$ spans a connected subtree of $\alttree$.
\end{enumerate*}
The \textit{width} of a tree decomposition equals $\max_{i \in I}|B_i|-1$, and the \emph{treewidth} $\tw{G}$ is the smallest width of any tree decomposition of $G$.

Similarly, the \emph{pathwidth} \cite{robertson1983graph} $\pw{G}$ is the minimum width of any {\em path decomposition} of $G$, i.e., a tree decomposition for which the tree $\alttree$ is required to be a path.

\paragraph*{Cutwidth.} Consider an ordering $(v_1, \ldots , v_n)$ of $V$. The set $C_{\ell} = \{\{v_i,v_j\} \in E : i \leq \ell < j\}$, where $1 \leq \ell < n$, is called a {\em cutset}. The {\em width} of the ordering is the size of the largest cutset. The {\em cutwidth} \cite{diaz2002survey}, denoted by $\cw{G}$, is the minimum width over all orderings of $V$.

While treewidth is useful in the analysis of algorithms \cite{bodlaender2008combinatorial}, cutwidth and, more generally, congestion (also known as carving-width) have recently turned out to be helpful mediators to connect treewidth with classical topological invariants \cite{tw-knots,HSWTreewidth_Arxiv,tw-hyperbolic}. 

\subsection{Triangulations and Heegaard splittings of 3-manifolds}
\label{ssec:3mfds}

The main objects of study in this article are {\em $3$-manifolds}, i.e., topological spaces in which every point has a neighborhood homeomorphic to $\R^3$ or to the closed upper half-space $\{(x,y,z)\in \R^3:z \geq 0\}$. For a 3-manifold $\manifold$, its {\em boundary} $\partial\manifold$ consists of all points of $\manifold$ not having a neighborhood homeomorphic to $\R^3$. A 3-manifold is {\em closed} if it is compact and its boundary is empty.
Two 3-manifolds are considered equivalent if they are homeomorphic.
We refer to \cite{schultens2014introduction} for an introduction to 3-manifolds (cf.\ \cite{hatcher3mfds}, \cite{hempel}, \cite{jaco} and \cite{thurston2014three}), and to \cite[Lecture 1]{saveliev} for an overview of the key concepts defined in this subsection.

All 3-manifolds considered in this paper are compact and orientable.

\paragraph*{Triangulations.} In the field of computational topology, a 3-manifold is often presented as a {\em triangulation} \cite{bing,moise}, i.e., a finite collection of abstract tetrahedra ``glued together'' by identifying pairs of their triangular faces called {\em triangles}. Due to these {\em face gluings}, several tetrahedral edges (or vertices) are also identified and we refer to the result as a single {\em edge (or vertex) of the triangulation}. The face gluings, however, cannot be arbitrary. For a triangulation $\tri$ to describe a closed $3$-manifold, it is necessary and sufficient that no tetrahedral edge is identified with itself in reverse, and the boundary of a small neighborhood around each vertex is a $2$-sphere. If, in addition, the boundaries of small neighborhoods of some of the vertices are disks, then $\tri$ describes a $3$-manifold with boundary.

To study a triangulation $\tri$, it is often useful to consider its {\em dual graph $\Gamma (\tri)$}, whose nodes and arcs correspond to the tetrahedra of $\tri$ and to the face gluings between them, respectively. By construction, $\Gamma (\tri)$ is a multigraph with maximum degree $\leq 4$.

\paragraph*{Heegaard splittings.} \emph{Handlebodies}, which can be thought of as thickened graphs, provide another way to describe 3-manifolds.
A {\em Heegaard splitting} \cite{scharlemann2002heegaard} is a decomposition $\manifold = \hbody \cup_f \hbody'$, where we start with the disjoint union of two homeomorphic handlebodies, $\hbody$ and $\hbody'$ and then identify their boundary surfaces via a homeomorphism $f\colon\partial\hbody \rightarrow \partial\hbody'$ referred to as the {\em attaching map}.
Every closed, compact, and orientable 3-manifold $\manifold$ can be obtained this way. Moreover, we may assume, without loss of generality, that $f$ is orientation-preserving. The smallest genus of a boundary surface ranging over all Heegaard splittings of $\manifold$, denoted by $\mathfrak{g}(\manifold)$, is called the {\em Heegaard genus} of $\manifold$.
Heegaard splittings with isotopic attaching maps yield homeomorphic 3-manifolds, hence are considered equivalent.

\subsection{Orientable Seifert fibered spaces}
\label{ssec:seifert}

Seifert fibered spaces, see \cite{seifert}, comprise an important class of 3-manifolds. Here we describe the orientable ones following \cite{saveliev} (cf.\ \cite[Sec.\ 2.1]{hatcher3mfds}, \cite[Ch.\ VI]{jaco}, \cite{montesinos}, \cite{orlik2006seifert}, or \cite[Sec.\ 3.7]{schultens2014introduction}).

Let us consider the surface $\altsurface_{g,r} = \altsurface_g \setminus (\operatorname{int}D_1 \cup \cdots \cup \operatorname{int}D_r)$ obtained from the closed orientable genus $g$ surface by removing the interiors of $r$ pairwise disjoint disks. Taking the product with the circle $\nsphere{1}$ yields an orientable 3-manifold $\altsurface_{g,r} \times \nsphere{1}$ whose boundary consists of $r$ tori; namely, $\partial(\altsurface_{g,r} \times \nsphere{1}) = (\partial D_1) \times \nsphere{1} \cup \cdots \cup (\partial D_r) \times \nsphere{1}$. For each $(\partial D_i) \times \nsphere{1}$, $1 \leq i \leq r$, we glue in a solid torus so that its meridian wraps $a_i$ times around the meridian $(\partial D_i) \times \{y_i\}$ and $b_i$ times around the longitude $\{x_i\} \times \nsphere{1}$ of $(\partial D_i) \times \nsphere{1}$. Here $a_i$ and $b_i$ are assumed to be coprime integers with $a_i \geq 2$, and the point $(x_i,y_i) \in (\partial D_i) \times \nsphere{1}$ is chosen arbitrarily. This way we obtain a closed orientable 3-manifold $\manifold = \sfs[\altsurface_g : (a_1,b_1),\ldots,(a_r,b_r)]$
which is called the {\em Seifert fibered space over $\altsurface_g$ with $r$ exceptional (or singular) fibers}. In relation to $\manifold$, the surface $\altsurface_g$ is referred to as the {\em base space} (or {\em orbit surface}).

\begin{example} 
{\em Lens space}s, the 3-manifolds of Heegaard genus one, coincide with Seifert fibered spaces over $\nsphere{2}$ having at most one (or two, cf.\ \cite[p. 27]{saveliev}) exceptional fiber(s).\footnote{In particular, we regard $\nsphere{2} \times \nsphere{1}$ (the SFS over $\nsphere{2}$ without exceptional fibers) to be a lens space as well.}
\label{ex:lens}
\end{example}

\paragraph*{Non-orientable base spaces.} With a slight modification of the above construction, one can obtain additional orientable Seifert fibered spaces having non-orientable base spaces. Beginning with $\nsurface_{g}$, the non-orientable genus $g$ surface, we pass to $\nsurface_{g,r}$ by adding $r$ punctures (i.e., by removing $r$ pairwise disjoint open disks). At this point, however, instead of taking the product $\nsurface_{g,r} \times \nsphere{1}$ (which yields a non-orientable 3-manifold) we consider the ``orientable $\nsphere{1}$-bundle'' over $\nsurface_{g,r}$, which has again $r$ torus boundary components. As before, we conclude by gluing in $r$ solid tori, specified by pairs of coprime integers $(a_i,b_i)$ with $a_i \geq 2$, where $1 \leq i \leq r$. The notation for the resulting 3-manifold remains the same.

\medskip

See \cite[Section 2]{lutz2003triangulations} for a concrete and detailed description of Seifert fibered spaces both over orientable and non-orientable surfaces (cf. the classes $\{Oo,g\}$ and $\{On,g\}$ therein).

\paragraph*{Geometric structures on 3-manifolds.} The significance of Seifert fibered spaces is exemplified by their role in the geometrization of 3-manifolds---a celebrated program initiated by Thurston \cite{thurston1982three}, influenced by Hamilton \cite{hamilton1982ricci}, and completed by Perelman \cite{perelman2002entropy,perelman2003ricci, perelman2003finite}, cf.\ \cite{Bessieres:Geometrisation-of-3-manifolds-2010,kleiner08-perelman, milnor2003, porti2008geometrization}---as they account for six out of the eight possible ``model geometries'' \cite{scott1983geometries} the building blocks may admit in the ``canonical decomposition'' of a closed 3-manifold.

\section{The treewidth of a 3-manifold}
\label{sec:tw3mfd}

In this section we prove \Cref{thm:cutwidth}. For this, we first recall how to turn graph-theoretical parameters, such as treewidth or cutwidth, into topological invariants of 3-manifolds. This is followed by a very brief and selective introduction to the theory {\em layered triangulations} as defined by Jaco and Rubinstein \cite{Jaco06Layering}. We then present the proof of \Cref{thm:cutwidth} which, on the topological level, is a direct consequence of this theory, and conclude with a remark on some practical aspects derived from the constructive nature of the proof.

\subsection{Topological invariants from graph parameters}

Recall the notions of treewidth and cutwidth from \Cref{ssec:graphs}.

\begin{definition}
\label{def:mfdparams}
Let $\manifold$ be a $3$-manifold and let $\tri$ be a triangulation of $\manifold$. By the {\em treewidth} of $\tri$ we mean $\tw{\dual(\tri)}$, i.e., the treewidth of its dual graph, and the {\em treewidth} $\tw{\manifold}$ of $\manifold$ is defined to be the smallest treewidth of any triangulation of $\manifold$. In other words,
\begin{align}
\label{eq:tw_mfd}
	\tw{\manifold} = \min\{\tw{\dual(\tri)}:\tri~\text{is~a~triangulation~of}~\manifold\}.
\end{align}
The definition of {\em cutwidth} $\cw{\manifold}$ and of {\em pathwidth} $\pw{\manifold}$ is analogous.
\end{definition}

From the definitions and \cite[Theorem 47]{bodlaender1998partial} it follows that $\tw{\manifold} \leq \pw{\manifold} \leq \cw{\manifold}$. 

Complementing \Cref{def:mfdparams}, we note that there are simple arguments proving that any given 3-manifold admits triangulations of arbitrarily high treewidth (cf.\ \Cref{app:high_tw}).

\subsection{Layered triangulations}
\label{ssec:layered}

The theory of layered triangulations of $3$-manifolds, due to Jaco and Rubinstein \cite{Jaco06Layering}, captures the inherently topological notion of a Heegaard splitting (see \Cref{ssec:3mfds}) in a combinatorial way. Here we outline the terminology important for our purposes. Despite all the technicalities, the nomenclature is very expressive and encapsulates much of the intuition.

\paragraph*{Spines and layerings.} Let $\nsurface_{g,r}$ denote the non-orientable surface of genus $g$ with $r$ {\em punctures} (i.e., boundary components).
A {\em $g$-spine} is a $1$-vertex triangulation of $\nsurface_{g,1}$.
It has one vertex, $3g-1$ edges (out of which $3g-2$ are interior and one is on the boundary), and $2g-1$ triangles. In particular, the Euler characteristic of any $g$-spine equals $1-g$.

\begin{figure}[ht]
	\centering
	\includegraphics[scale=1]{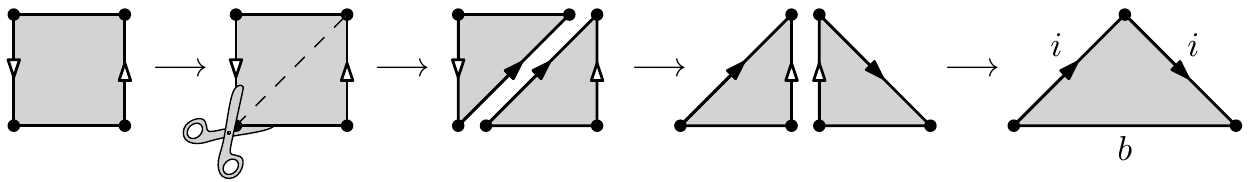}
	\caption{Left to right: Transforming the (well-known depiction of the) M\"obius band---the non-orientable surface of genus one with one puncture---into a $1$-spine with interior edge $i$.}
	\label{fig:one_trg_moeb}
\end{figure}

Now consider a triangulation $\surface$ of a surface---usually seen as a $g$-spine or as the boundary of a triangulated $3$-manifold---and let $e$ be an interior edge of $\surface$ with $t_1$ and $t_2$ being the two triangles of $\surface$ containing $e$. Gluing a tetrahedron $\tetrahedron$ along $t_1$ and $t_2$ without a twist is called a {\em layering} onto the edge $e$ of the surface $\surface$, cf.\ \Cref{fig:layering}(i).
Importantly, we allow $t_1$ and $t_2$ to coincide, e.g., when layering on the interior edge of a 1-spine (\Cref{fig:one_trg_moeb}, right).

\begin{figure}[ht]
	\centering
	\includegraphics[scale=1]{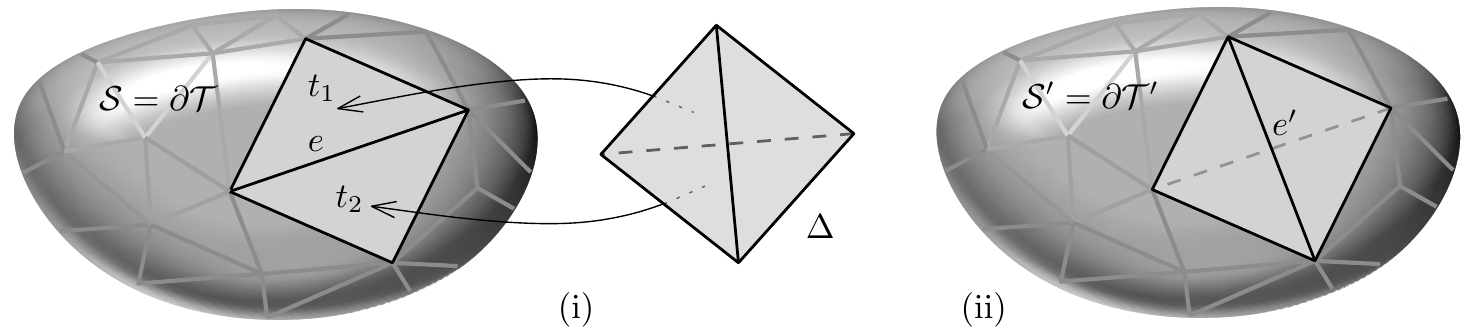}
	\caption{(i) Layering onto the edge $e$ of $\surface = \partial \tri$, (ii) which has the effect of ``flipping'' $e$.}
	\label{fig:layering}
\end{figure}

\paragraph*{Layered handlebodies.}
It is a pleasant fact that by layering a tetrahedron onto each of the $3g-2$ interior edges of a $g$-spine we obtain a triangulation of the genus $g$ handlebody $\hbody_g$, called a {\em minimal layered triangulation} thereof (see \Cref{fig:spine}). More generally, we call any triangulation obtained by additional layerings a {\em layered triangulation} of $\hbody_g$ \cite[Section 9]{jaco}.

The case $g=1$ is of particular importance. Starting with a 1-spine (\Cref{fig:one_trg_moeb}) and layering on its interior edge $i$ produces a $1$-tetrahedron triangulation $\tri$ of the solid torus $\hbody_1$ (\Cref{app:torus}). Its boundary $\surface = \partial\tri$ is the unique $2$-triangle triangulation of the torus with one vertex and three edges, and layering onto any edge of $\surface$ yields another triangulation of $\hbody_1$. We may iterate this procedure to obtain further triangulations, any of which we call a {\em layered solid torus} (cf.\ \cite[Section 1.2]{burton2003minimal} for a detailed exposition). By construction, its dual graph consists of a single loop at the end of a path of double arcs; see, e.g, \Cref{fig:meridian}(v).

\begin{figure}[ht]
	\centering
	\includegraphics[scale=0.406]{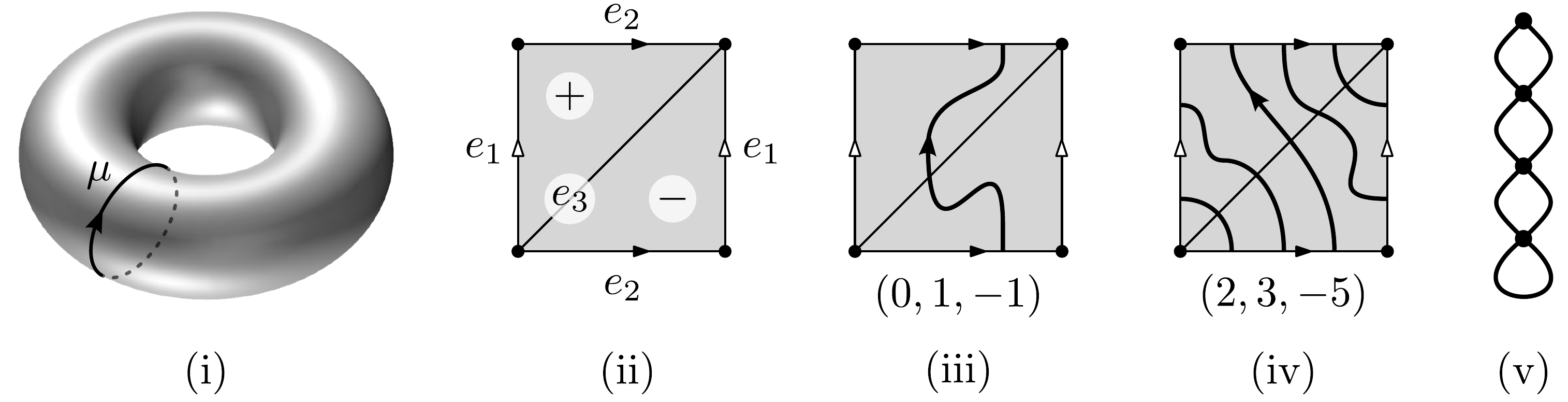}
	\caption{\label{fig:meridian}}
\end{figure}

While combinatorially the same, boundary triangulations of layered solid tori generally are not isotopic; they can be described as follows. Consider a ``reference torus'' $\mathbb{T}$ with a fixed meridian $\mu$, \Cref{fig:meridian}(i). Given a layered solid torus, its boundary induces a triangulation of $\mathbb{T}$. Label the two triangles with \scalebox{1.2}{$+$} and \scalebox{1.2}{$-$}, and the three edges with $e_1,e_2$, and $e_3$, \Cref{fig:meridian}(ii); and fix an orientation of $\mu$. Let $p,q$ and $r$ denote the geometric intersection number of $\mu$ with $e_1$, $e_2$ and $e_3$, respectively. We say that the corresponding layered solid torus is of type $(p,q,r)$, or $\lst (p,q,r)$ for short. See, e.g., \Cref{fig:meridian}(iii)--(iv).

It can be shown that $p,q,r$ are always coprime with $p+q+r = 0$. Conversely, for any such triplet, one can construct a layered solid torus of type $(p,q,r)$, cf.\ \cite[Algorithm 1.2.17]{burton2003minimal}.

\paragraph*{Layered 3-manifolds.} Let $\manifold$ be a closed, orientable $3$-manifold given via a Heegaard splitting $\manifold = \hbody \cup_f \hbody'$. If $\hbody$ and $\hbody '$ can be endowed with layered triangulations $\tri$ and $\tri'$, respectively, such that the attaching map $f$ is simplicial (i.e., respects the triangulations of $\partial \tri$ and $\partial \tri'$), then the union $\tri \cup_{f} \tri'$ triangulates $\manifold$ and is called a {\em layered triangulation} of $\manifold$. The next theorem is fundamental.

\begin{theorem}[Jaco--Rubinstein, Theorem 10.1 of \cite{Jaco06Layering}]
Every closed, orientable 3-manifold admits a layered triangulation (which is a one-vertex triangulation by construction).
\label{thm:layered3Mfd}
\end{theorem}

\subsection{Treewidth versus Heegaard genus}
\label{ssec:treewidth3mfd}

In \cite[Theorem 4]{HSWTreewidth_Arxiv} it is shown that for a closed, orientable, irreducible, non-Haken (cf. \cite[Section 2.2]{HSWTreewidth_Arxiv}) $3$-manifold $\manifold$, the Heegaard genus $\mathfrak{g}(\manifold)$ and the treewidth $\tw{\manifold}$ satisfy
\begin{align}
\label{eq:heeg-tw}
	\mathfrak{g}(\manifold) < 24(\tw{\manifold}+1).
\end{align}
In this section we further explore the connection between these two parameters, guided by two questions:
\begin{enumerate*}
	\item Does a reverse inequality hold?
	\item Can one refine the assumptions?
\end{enumerate*}

For the first one, we give an affirmative answer (\Cref{thm:cutwidth}). The result is almost immediate if one inspects a layered triangulation of a closed, orientable 3-manifold. Due to work of Jaco and Rubinstein, this approach is always possible (cf. \Cref{thm:layered3Mfd}).

The second question is more open-ended. As a first step, we observe the following.

\begin{prop}
\label{prop:bounded}
There exists an infinite family of 3-manifolds of bounded cutwidth---hence of bounded treewidth---with unbounded Heegaard genus.
\end{prop}

\begin{proof}
Consider $\{\manifold_k = \manifold^{\# k} : k \in \mathbb{N}\}$, where $\manifold_k$ is the $k$-fold connected sum of a $3$-manifold $\manifold$ of Heegaard genus one (e.g., $\manifold = \mathbb{S}^2 \times \mathbb{S}^1$). Since the Heegaard genus is additive under taking connected sums \cite[Corollary II.10.]{jaco}, we have that $\mathfrak{g}(\manifold_k)=k$.
  
However, $\manifold_k$ admits a triangulation of bounded cutwidth. Indeed, start with a fixed triangulation $\tri$ of $\manifold$ containing two tetrahedra $\tetrahedron_1$ and $\tetrahedron_2$ which 
\begin{enumerate*}[a)]
	\item do not share any vertices in $\tri$, and
	\item do not have any self-identifications in $\tri$.
\end{enumerate*}

Now let $\width$ denote the width of an ordering of $V(\dual(\tri))$---the nodes of the dual graph $\dual(\tri)$---in which $\tetrahedron_1$ and $\tetrahedron_2$ correspond to the first and the last node, respectively. Moreover, let $\tri^{(i)}$ $(1 \leq i \leq k)$ be $k$ copies of $\tri$. Forming connected sums along $\tetrahedron_2^{(i)}$ and $\tetrahedron_1^{(i+1)}$ $(1 \leq i \leq k-1)$ yields a triangulation $\tri_k$ of $\manifold_k$ together with an ordering of $V(\dual(\tri_k))$ of width $\width$, see \Cref{fig:concatenate}. Therefore $\cw{\manifold_k} \leq \cw{\tri_k} \leq \width$ for every $k \in \mathbb{N}$.
\end{proof}

\begin{figure}[ht]
	\centering
	\includegraphics[scale=1]{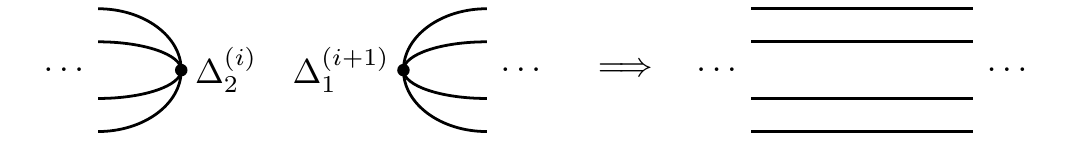}
	\caption{The effect of forming $\tri^{(i)} \# \tri^{(i+1)}$ at the level of the dual graphs.}
	\label{fig:concatenate}
\end{figure}

\begin{remark}
\Cref{prop:bounded} shows that among reducible 3-manifolds one can easily find infinite families which violate \eqref{eq:heeg-tw}. Nevertheless, irreducibility alone is insufficient for \eqref{eq:heeg-tw} to hold. In particular, in \Cref{sec:tw_two} we prove that orientable Seifert fibered spaces over $\mathbb{S}^2$ have treewidth at most two (\Cref{thm:sfs1}). However, all but two of them are irreducible \cite[Theorem 3.7.17]{schultens2014introduction} and they can have arbitrarily large Heegaard genus \cite[Theorem 1.1]{boileau1984heegaard}.

Recent work of de Mesmay, Purcell, Schleimer, and Sedgwick \cite{tw-knots} suggests that one might be able to obtain an inequality similar to \eqref{eq:heeg-tw} for (closed) Haken manifolds as well, by imposing appropriate conditions on the (incompressible) surfaces they contain.
\end{remark}

Nevertheless, as mentioned before, a reverse inequality always holds.

\begin{proof}[Proof of \Cref{thm:cutwidth}]
Let $g = \mathfrak{g}(\manifold)$. Consider the $g$-spine $\spine$ in \Cref{fig:spine}(i) together with the indicated order in which we layer onto the $3g-2$ interior edges of $\spine$ to build two copies $\tri'$ and $\tri''$ of a minimal layered triangulation of the genus $g$ handlebody. See \Cref{fig:spine}(ii) for the dual graph of $\tri'$ (and of $\tri''$). Note that $\partial\tri'$ and $\partial\tri''$ consist of $4g-2$ triangles each.

\begin{figure}
	\centering
    \includegraphics[scale=1.0909]{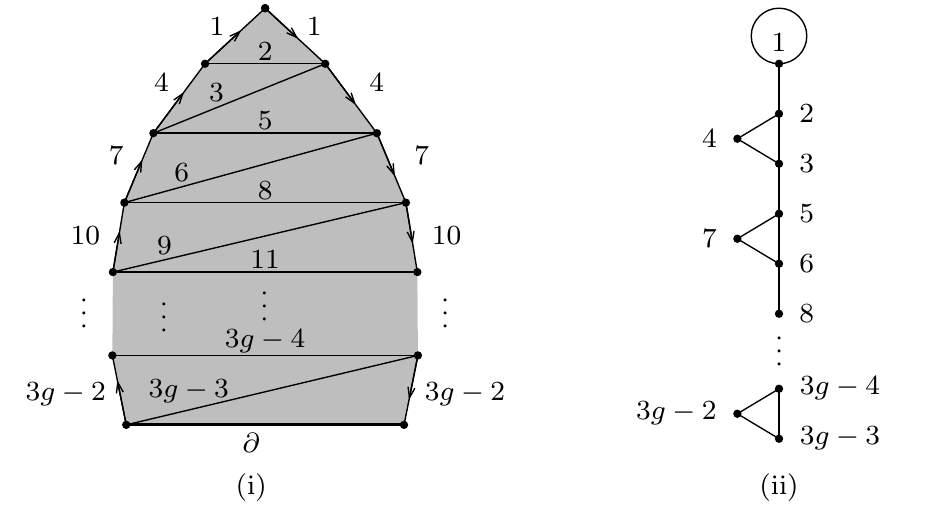}
    \caption{A $g$-spine $\spine$ together with the order in which we layer onto its interior edges (i). The dual graph of resulting minimal layered triangulation of the genus $g$ handlebody (ii).}
    \label{fig:spine}
\end{figure} 

By \Cref{thm:layered3Mfd}, we may extend $\tri'$ to a layered triangulation $\tri'''$ which can be glued to $\tri''$ along a simplicial map $f: \partial \tri''' \to \partial\tri''$ to yield a triangulation $\tri = \tri''' \cup_f \tri''$ of $\manifold$. This construction imposes a natural ordering on the tetrahedra of $\tri$:
\begin{enumerate*}
    \item Start by ordering the tetrahedra of $\tri'$ according to the labels of the edges of $\spine$ onto which they are initially layered.
    \item Continue with all tetrahedra between $\tri'$ and $\tri''$ in the order they are attached to $\tri'$ in order to build up $\tri'''$.
    \item Finish with the tetrahedra of $\tri''$ again in the order of the labels of the edges of $\spine$ onto which they are layered.
\end{enumerate*}
This way we obtain a linear layout of the nodes of $\dual (\tri)$ which realizes width $4g-2$ (\Cref{fig:layout}). Therefore $\cw{\manifold} \leq 4g-2$.
\end{proof}

\begin{figure}[ht]
	\centering
    \includegraphics[scale=1.0909]{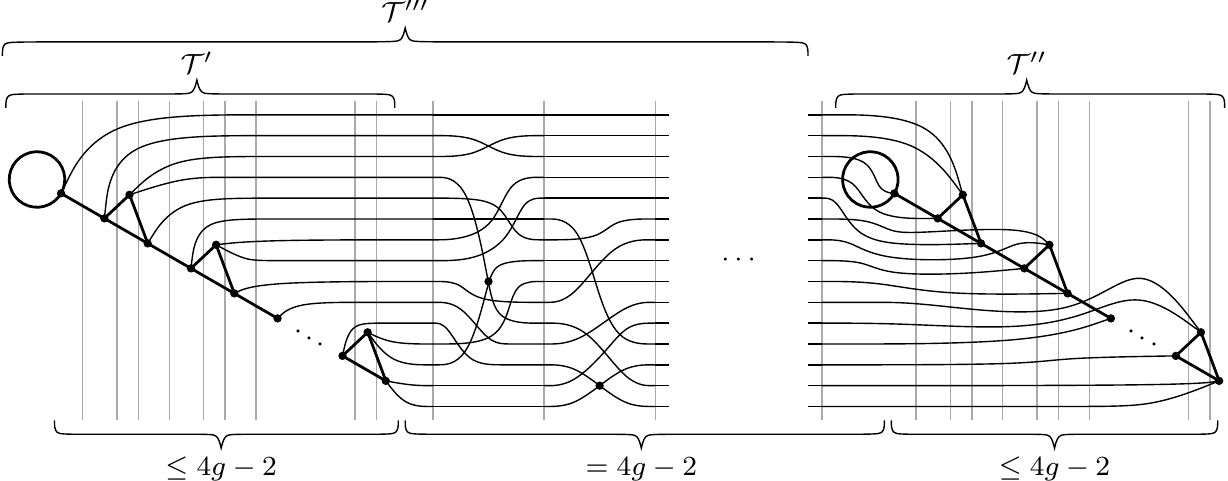}
    \caption{A linear layout showing that $\cw{\manifold}$ is bounded above by $4 \mathfrak{g}(\manifold) - 2$.}
    \label{fig:layout}
\end{figure}

Combining \Cref{thm:cutwidth} with $\tw{\manifold} \leq \cw{\manifold}$, we directly deduce the following.

\begin{cor}
\label{cor:nonhaken}
For any closed, orientable, irreducible, non-Haken $3$-manifold $\manifold$ the Heegaard genus $\mathfrak{g}(\manifold)$ and the treewidth $\tw{\manifold}$ satisfy
\begin{align}
    \frac14 \tw{\manifold} + 2 \leq \mathfrak{g}(\manifold ) < 24 (\tw{\manifold} +1).
	\label{eq:tw-heeg-tw}
\end{align}
\end{cor}

In \cite[Question 5.3]{Bachmann17HeegaardGenus} the authors ask whether computing the Heegaard genus of a $3$-manifold is still hard when restricting to the set of non-Haken $3$-manifolds. \Cref{cor:nonhaken} implies that the answer to this question also has implications on the hardness of computing or approximating the treewidth of non-Haken manifolds.

\subsection{An algorithmic aspect of layered triangulations}
\label{ssec:application}

  Layered triangulations are intimately related to the rich theory of surface homeomorphisms, and in particular the notion of the mapping class group. Making use of this connection, as well as some results due to Bell \cite{Bell15Diss}, we present a general algorithmic method to turn a $3$-manifold $\manifold$, given by a small genus Heegaard splitting in some reasonable way, into a triangulation of $\manifold$ while staying in full control over the size of this triangulation.

  Namely, if $\manifold$ is given by a genus $g$ Heegaard splitting with the attaching map presented as a word $w$ over a set of Dehn twists $\gens$ generating the genus $g$ mapping class group, then 
there exists a constant $K(g,X)$ such that we can construct a layered triangulation of $\manifold$ of size $O(K(g,X)|w|)$, cutwidth $\leq 4g-2$, in time $O\left (K(g,X) (|\gens| + |w|)\right )$.
 See \Cref{app:application} for further explanations, a precise formulation of the above statement, and a proof.

\section{3-Manifolds of treewidth one}
\label{sec:tw_one}

This section is dedicated to the proof of \Cref{thm:main}. As the treewidth is not sensitive to multiple arcs or loops, it is helpful to also consider {\em simplifications} of multigraphs, in which we forget about loops and reduce each multiple arc to a single one (\Cref{fig:simplification}).

\begin{figure}[ht]
	\centering
	\includegraphics[scale=1]{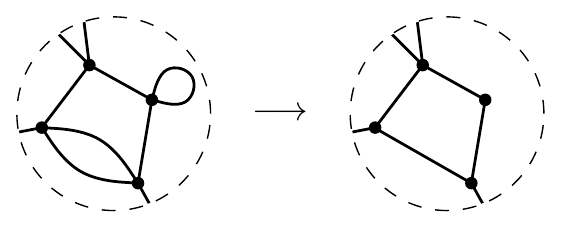}
	\caption{The local effect of simplification in a multigraph.}
	\label{fig:simplification}
\end{figure}

One direction in \Cref{thm:main} immediately follows from work of Jaco and Rubinstein.

\begin{theorem}[cf.\ Theorem 6.1 of \cite{Jaco06Layering}]
\label{thm:layeredlensspaces}
Every lens space admits a layered triangulation $\tri$ with the simplification of $\dual (\tri)$ being a path. In particular, all $3$-manifolds of Heegaard genus at most one have treewidth at most one.
\end{theorem}

For the proof of the other direction, the starting point is the following observation.

\begin{lemma}
If the simplification of a $4$-regular multigraph $G$ is a tree, then it is a path.
\label{lem:simplification}
\end{lemma}

\begin{proof}
Let $S(G)$ denote the simplification of $G$. Call an arc of $S(G)$ {\em even} (resp. {\em odd}) if its corresponding multiple arc in $G$ consist of an even (resp. odd) number of arcs. Let $Odd(G)$ be the subgraph of $S(G)$ consisting of all odd arcs. It follows from a straightforward parity argument that all nodes in $Odd(G)$ have an even degree.
In particular, if the set $E(Odd(G))$ of arcs is nonempty, then it necessarily contains a cycle. However, this cannot happen as $S(G)$ is a tree by assumption. Consequently, all arcs of $S(G)$ must be even.
This implies that every node of $S(G)$ has degree at most $2$ (otherwise there would be a node in $G$ with degree $> 4$), which in turn implies that $S(G)$ is a path.
\end{proof}

\begin{figure}[ht]
	\centering	
	\includegraphics[scale=1]{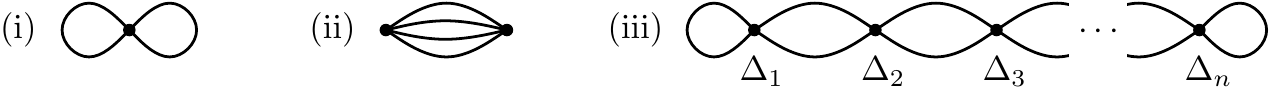}
    \caption{The only possible dual graphs (corresponding to closed 3-manifolds) of treewidth at most one.}
    \label{fig:thickpath}
\end{figure}

Consequently, if $\tw{\dual(\tri)}\leq1$ for a triangulation $\tri$ of a closed 3-manifold, then $\dual(\tri)$ is a ``thick'' path. If $\dual(\tri)$ has only one node, then it has two loops (\Cref{fig:thickpath}(i)). By looking at the {\em Closed Census} \cite{Regina}, we see that the only orientable 3-manifolds admitting a dual graph of this kind are $\nsphere{3}$ and two lens spaces. If $\dual(\tri)$ has a quadruple arc, then it must be a path of length two (\Cref{fig:thickpath}(ii)), and the only 3-manifold not a lens space appearing here is  $\sfs[\nsphere{2} : (2,1),(2,1),(2,-1)]$. 
Otherwise, order the tetrahedra $\Delta_1 , \ldots , \Delta_n$ of $\tri$ as shown in \Cref{fig:thickpath}(iii), and define $\tri_i \subset \tri$ to be the {\em $i$\textsuperscript{th} subcomplex} of $\tri$ consisting of $\Delta_1, \ldots , \Delta_i$.

$\tri_1$ is obtained by identifying two triangles of $\Delta_1$. Without loss of generality, we may assume that these are the triangles $\Delta_1 (013)$ and $\Delta_1 (023)$. A priori, there are six possible face gluings between them (corresponding to the six bijections $\{0,1,2\}\rightarrow\{0,2,3\}$).

The gluing $\Delta_1 (013) \mapsto \Delta_1 (023)$ yields a $3$-vertex triangulation of the $3$-ball, called a {\em snapped $3$-ball}, and is an admissible choice for $\tri_1$, \Cref{fig:lst}(i). $\Delta_1 (013) \mapsto \Delta_1 (032)$ and $\Delta_1 (013) \mapsto \Delta_1 (203)$ both create M\"obius bands as vertex links of the vertices $(0)$ and $(2)$, respectively, and thus these 1-tetrahedron complexes cannot be subcomplexes of a $3$-manifold triangulation. $\Delta_1 (013) \mapsto \Delta_1 (230)$ and $\Delta_1 (013) \mapsto \Delta_1 (302)$ both produce valid but isomorphic choices for $\tri_1$: the minimal layered solid torus of type $\lst(1,2,-3)$, \Cref{fig:lst}(ii). Lastly, $\Delta_1 (013) \mapsto \Delta_1 (320)$ identifies the edge $(03)$ with itself in reverse, it is hence invalid.

We discuss the two valid cases separately, starting with the latter one.

\begin{figure}[htb]
\centering
	\includegraphics[scale=1]{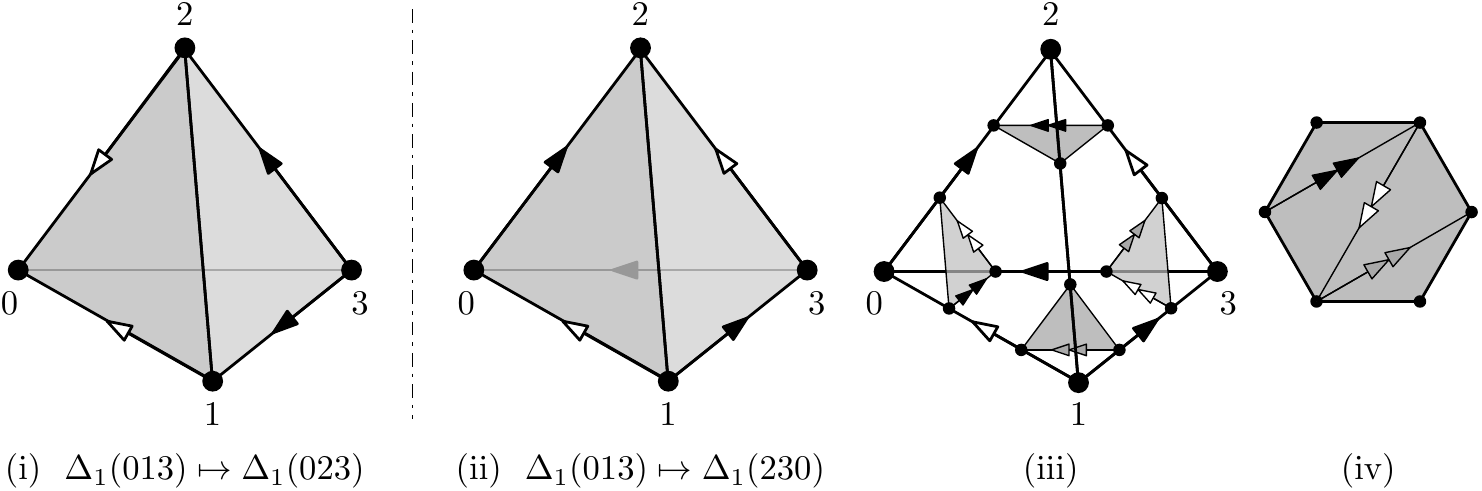}
    \caption{The snapped 3-ball (i). A layered solid torus (ii), with four normal triangles comprising the single vertex link (iii), which is a triangulated hexagonal disk (iv). \label{fig:lst}}
\end{figure}

\begin{lemma}
\label{lem:lst}
Let $\tri$ be a triangulation of a closed, orientable $3$-manifold of treewidth one, with $\tri_1$ being a solid torus. Then $\tri$ triangulates a $3$-manifold of Heegaard genus one.
\end{lemma}

\begin{proof}
The proof consists of the following parts. 
\begin{enumerate*}
	\item We systematize all subcomplexes $\tri_2 \subset \tri$ which arise from gluing a tetrahedron $\Delta_2$ to $\tri_1$ along two triangular faces, and discard all complexes which cannot be part of a 3-manifold triangulation.
	\item For the remaining cases we discuss the combinatorial types of complexes $\tri_i$, $i > 2$, and triangulations of 3-manifolds arising from them.
	\item We show for all resulting triangulations, that the fundamental group of the underlying manifold is cyclic, and that it thus is of Heegaard genus at most one.
\end{enumerate*}

To enumerate all possibilities for $\tri_2$, assume, without loss of generality, that $\tri_1$ is obtained by $\Delta_1 (013) \mapsto \Delta_1 (230)$. The boundary $\partial\tri_1$ is built from two triangles $(012)_\partial$ and $(123)_\partial$, sharing an edge $(12)$, via the identifications $(01) = (23)$ and $(02)=(13)$, see \Cref{fig:lst}(ii). The vertex link of $\tri_1$ is a triangulated hexagon as shown in \Cref{fig:lst}(iii)--(iv).

The second subcomplex $\tri_2$ is obtained from $\tri_1$ by gluing $\Delta_2$ to the boundary of $\tri_2$ along two of its triangles. By symmetry, we are free to choose the first gluing. Hence, without loss of generality, let $\tri'_2$ be the complex obtained from $\tri_1$ by gluing $\Delta_2$ to $\tri_1$ with gluing $\Delta_2 (012) \mapsto (012)_\partial$. The result is a 2-vertex triangulated solid torus with four boundary triangles $\Delta_2 (013)$, $\Delta_2 (023)$, $\Delta_2 (123)$ and $(123)_\partial$, see \Cref{fig:ninegon}(ii). Since adjacent edges in the boundary of the link of $\tri_1$ are always normal arcs in distinct triangles of $\partial \tri_1$, the vertex links of $\tri'_2$ must be a triangulated 9-gon and a single triangle, shown in \Cref{fig:ninegon}(iii).

Note that both vertex links of $\tri'_2$ are symmetric with respect to the normal arcs coming from boundary triangles $\Delta_2 (013)$, $\Delta_2 (023)$ and $\Delta_2 (123)$. By this symmetry, we are free to chose whether to glue $\Delta_2 (013)$, $\Delta_2 (023)$ or $\Delta_2 (123)$ to $(123)_\partial$, in order to obtain $\tri_2$.

\begin{figure}[ht!]
	\centering
	\includegraphics[scale=1]{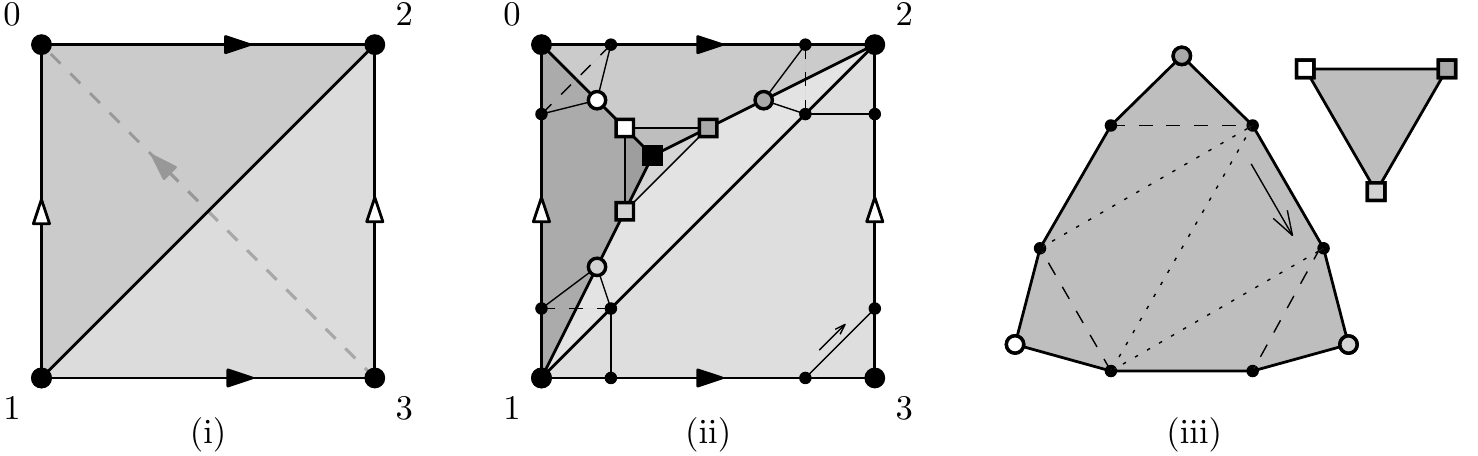}
    \caption{The solid torus $\tri_1$ (i), the complex $\tri'_2$ (ii), and the vertex links of $\tri'_2$ (iii). \label{fig:ninegon}}
  \end{figure} 
  
\begin{figure}[ht!]
	\centering
	\includegraphics[scale=1]{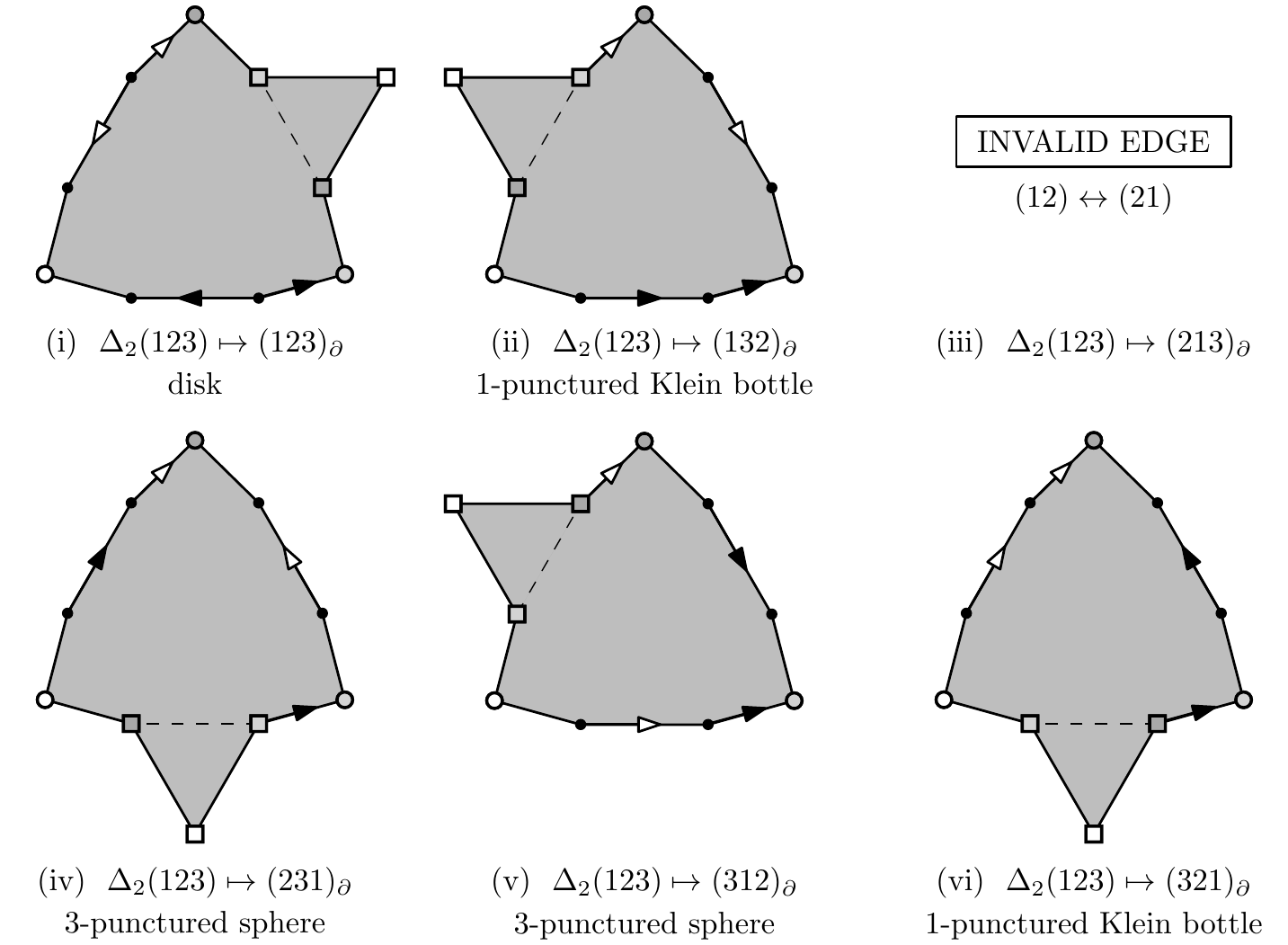}
	\caption{\label{fig:allcases}}
\end{figure} 

Therefore we have the following six possibilities to consider (\Cref{fig:allcases}).

\begin{description}
	\item[$ \Delta_2 (123) \mapsto (123)_\partial$] is a layering onto $(23)$. It yields another layered solid torus with vertex link a triangulated hexagon with edges adjacent in the boundary of the link being normal arcs in distinct faces in $\partial \tri_2$. Hence, in this case we have the same options for $\tri_3$ as the ones in this list. Any complex obtained by iterating this case is of this type. 

Here, as well as for the remainder of this proof, whenever we obtain a subcomplex with all cases for the next subcomplex equal to a case already considered (i.e., isomorphic boundary complexes compatible with isomorphic boundaries of vertex links), we talk about these cases to be of the same {\em type}. We denote the current one by $\mathscr{T}_{\text{I}}$.
	\item[$ \Delta_2 (123) \mapsto (132)_\partial$] is invalid, as it creates a punctured Klein bottle as vertex link.
	\item[$ \Delta_2 (123) \mapsto (213)_\partial$] is invalid, as it identifies $(12)$ on the boundary with itself in reverse.
	\item[$ \Delta_2 (123) \mapsto (231)_\partial$] results in a 1-vertex complex with triangles $(013)_\partial$ and $(023)_\partial$ comprising its boundary, which is isomorphic to that of the snapped 3-ball with all of its three vertices identified.
The vertex link is a 3-punctured sphere with two boundary components being normal loop arcs and one consisting of the remaining four normal arcs. This complex is discussed in detail below, and we denote its type by $\mathscr{T}_{\text{II}}$.
	\item[$ \Delta_2 (123) \mapsto (312)_\partial$] gives a 1-vertex complex of type $\mathscr{T}_{\text{II}}$ as in the previous case.
	\item[$ \Delta_2 (123) \mapsto (321)_\partial$] is invalid: it produces a punctured Klein bottle in the vertex link.
\end{description}

Now we discuss complexes of type $\mathscr{T}_{\text{II}}$. To this end, let $\tri_2$ be the complex in \Cref{fig:threepunctsphere}(ii) defining this type. By gluing $\Delta_3$ to $\tri_2$ along a boundary triangle, say $\Delta_3 (013) \mapsto (013)_\partial$, we obtain a complex $\tri'_3$ (see \Cref{fig:threepunctsphere}(iii)). Note that no boundary edge of the 3-punctured sphere vertex link $\mathcal{L}$ can be identified with an edge in another boundary component of $\mathcal{L}$, for that would create genus in $\mathcal{L}$ (an obstruction to being a subcomplex of a 3-manifold triangulation in which all vertex links must be 2-spheres). As shown in \Cref{fig:threepunctsphere_links}, there is a unique gluing to avoid this, namely $\Delta_3 (023) \mapsto (023)_\partial$, which yields a 1-vertex complex $\tri_3$ with vertex link still being a 3-punctured sphere, but now with three boundary components consisting of two edges each, as indicated in \Cref{fig:threepunctsphere_links}(i). Let $\mathscr{T}_{\text{III}}$ denote its type. Repeating the same argument for $\tri_3$ implies that a valid $\tri_4$ must be again of type $\mathscr{T}_{\text{II}}$.

\begin{figure}[ht!]
	\centering
	\includegraphics[scale=1]{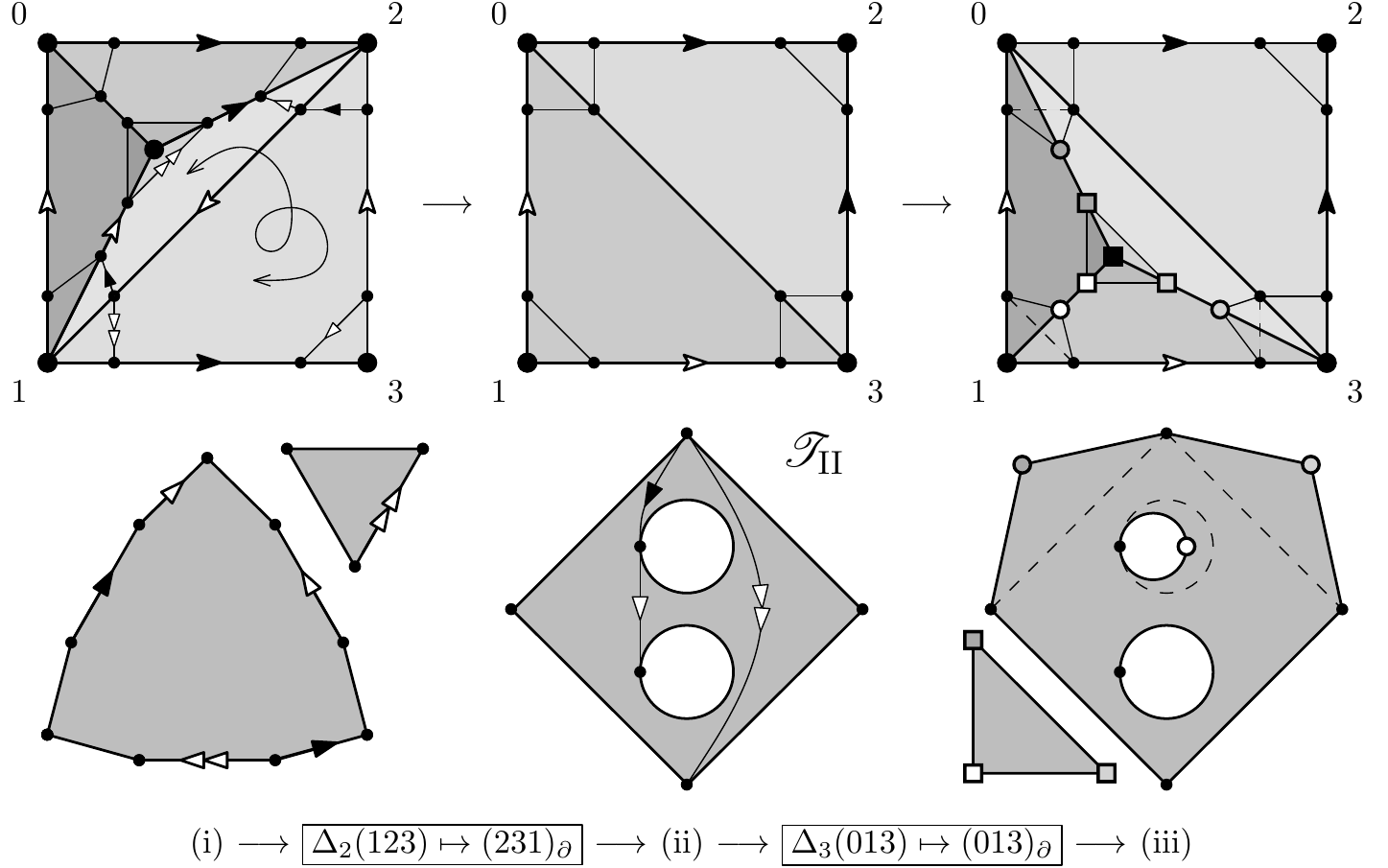}
    \caption{}
    \label{fig:threepunctsphere}
\end{figure} 

\begin{figure}[ht]
	\centering
	\includegraphics[scale=1]{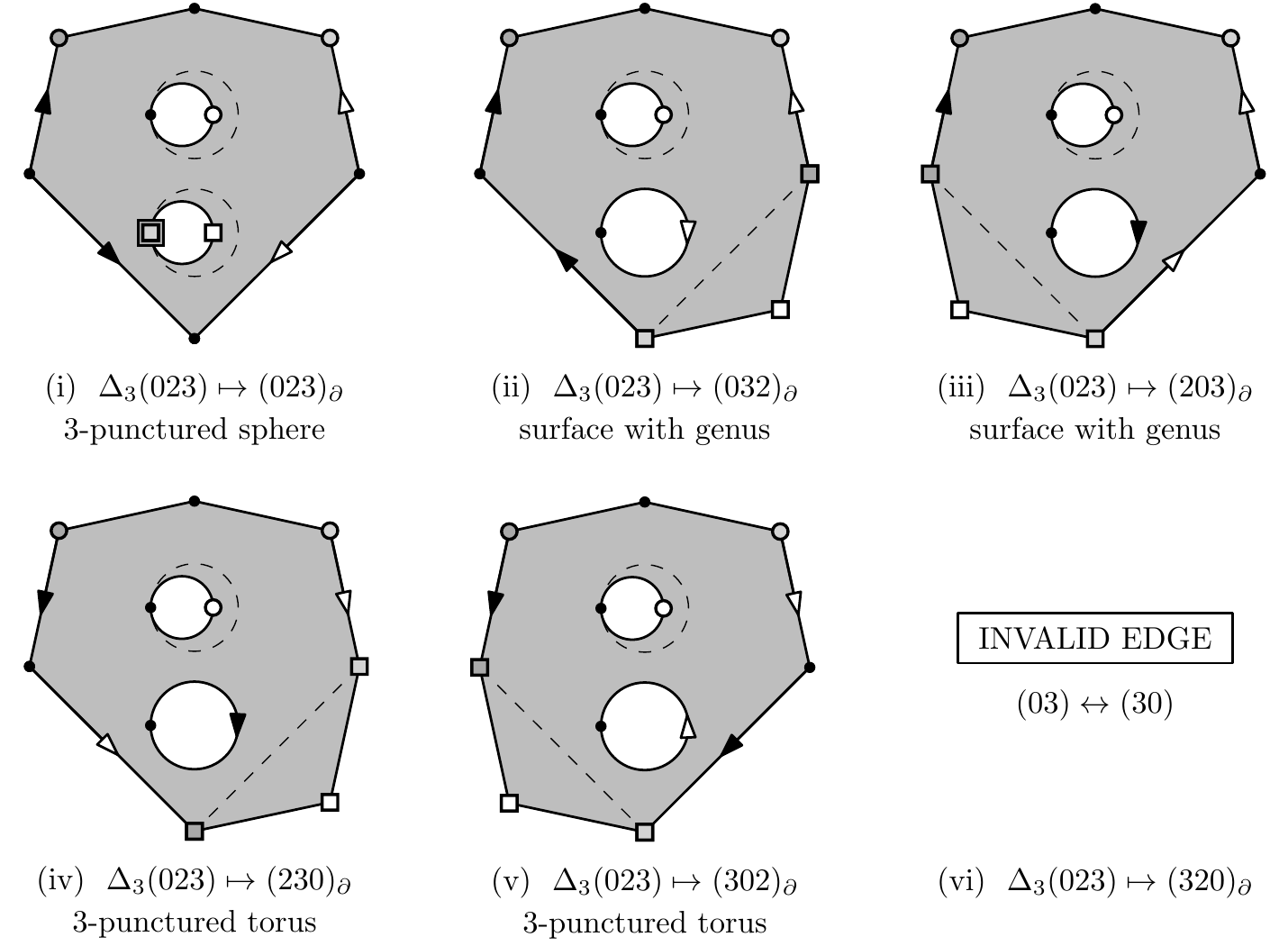}
    \caption{}
    \label{fig:threepunctsphere_links}
\end{figure}

Altogether, the type of each intermediate complex $\tri_i$ ($i < n$) is either $\mathscr{T}_{\text{I}}$ (a layered solid torus), or one of the two types $\mathscr{T}_{\text{II}}$ and $\mathscr{T}_{\text{III}}$ of 1-vertex complexes with a 3-punctured sphere as vertex link.
If $\tri_{n-1}$ is of type $\mathscr{T}_{\text{I}}$, then it can always be completed to a triangulation of a closed 3-manifold by either adding a minimal layered solid torus or a snapped 3-ball.
If $\tri_{n-1}$ is of type $\mathscr{T}_{\text{II}}$, it may be completed by adding a snapped 3-ball. If $\tri_{n-1}$ is of type $\mathscr{T}_{\text{III}}$ it cannot be completed to a triangulation of a $3$-manifold.

\medskip

To conclude that any resulting $\tri$ triangulates a 3-manifold of Heegaard genus at most one, first we observe that the fundamental group of $\pi_1(\tri)$ is generated by one element.

Indeed, $\pi_1(\tri_1)$ is isomorphic to $\mathbb{Z}$ and is generated by a boundary edge. Furthermore, since $\tri_1$ only has one vertex, all edges in $\tri_1$ must be loop edges, and no edge which is trivial in $\pi_1(\tri_1)$ can become non-trivial in the process of building up the triangulation of the closed 3-manifold.
When we extend $\tri_1$ by attaching further tetrahedra along two triangles each, then either all edges of the newly added tetrahedron are identified with edges of the previous complex, or---in case of a layering---the unique new boundary edge can be expressed in terms of the existing generator. In both cases, the fundamental group of the new complex admits a presentation with one generator. Moreover, no new generator can arise from inserting a minimal layered solid torus or snapped 3-ball in the last step.

\medskip

So either $\pi_1(\tri)$ is infinite cyclic, i.e., isomorphic to $\mathbb{Z}$, in which case $\tri$ must be a triangulation of $\nsphere{2} \times \nsphere{1}$ \cite{hatcher3mfds}; or $\pi_1(\tri)$ is finite, but then it is spherical by the Geometrization Theorem \cite[p.\ 104]{porti2008geometrization}, and thus must be a lens space \cite[Theorem 4.4.14.(a)]{thurston1982three}.
\end{proof} 

\begin{lemma}
\label{lem:ball}
Let $\tri$ be an $n$-tetrahedron triangulation of a closed, orientable $3$-manifold of treewidth one, with both $\tri_1$ and $\tri \setminus \tri_{n-1}$ being a snapped $3$-ball. Then $\tri$ triangulates a $3$-manifold of Heegaard genus one.
\end{lemma}

\begin{proof}
The proof follows the same structure as the one of \Cref{lem:lst}. Let $\tri_1$ be a snapped 3-ball, say, obtained by the gluing $\Delta_1 (013) \mapsto \Delta_1 (023)$. Its boundary is a three-vertex two-triangle triangulation of the 2-sphere with triangles $(012)_\partial$ and $(123)_\partial$ glued along common edge $(12)$ with edge identifications $(01) = (02)$ and $(13)=(23)$, see \Cref{fig:spd3b}(i). One vertex link of $\tri_1$ consists of two triangles identified along a common edge, and two vertex links are single triangles with two of their edges identified (\Cref{fig:spd3b}(iii)).

\begin{figure}[htb]
	\centering
	\includegraphics[scale=1]{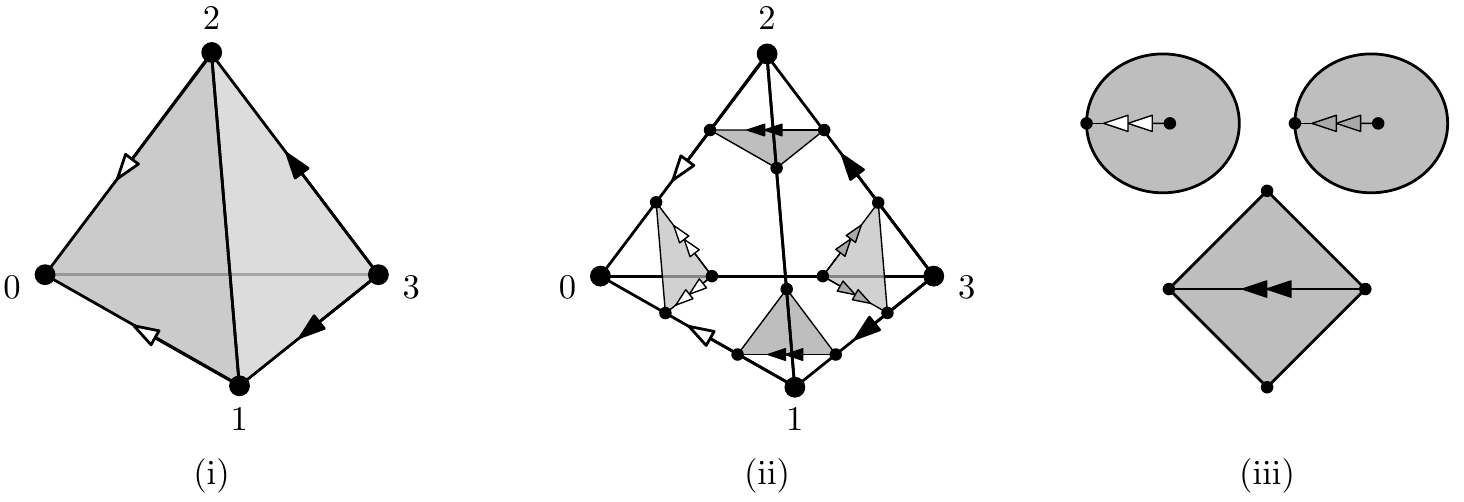}
	\caption{(i) A Snapped 3-ball, (ii) its normal triangles, and (iii) the three vertex links.}
	\label{fig:spd3b}
\end{figure}

$\tri_2$ is obtained from $\tri_1$ by gluing $\Delta_2$ to the boundary of $\tri_2$ along two of its triangles. Again, by symmetry, we are free to choose the first gluing. Hence, let $\tri'_2$ be the complex obtained from $\tri_1$ by gluing $\Delta_2$ to $\tri_1$ via $\Delta_2 (012) \mapsto (012)_\partial$. The result is a 4-vertex triangulated 3-ball with four boundary triangles $\Delta_2 (013)$, $\Delta_2 (023)$, $\Delta_2 (123)$, and $(123)_\partial$ (\Cref{fig:quadlk}(i)). The vertex links of $\tri'_2$ are a two-triangle triangulation of a bigon with an interior vertex of degree one, a triangulation of a hexagon, and two 1-triangle triangulations of a disk, one with a single boundary edge, and one with three boundary edges, cf.\ \Cref{fig:quadlk}(ii). While all four vertex links of $\tri'_2$ are symmetric with respect to the normal arcs of the boundary triangles $\Delta_2 (013)$ and $\Delta_2 (023)$, the normal arcs of $\Delta_2 (123)$ are different.

It follows that there are twelve possibilities to consider (see \Cref{fig:sp3b_cases_1,fig:sp3b_cases_2}).

\begin{figure}[htb]
 	\centering
	\includegraphics[scale=1]{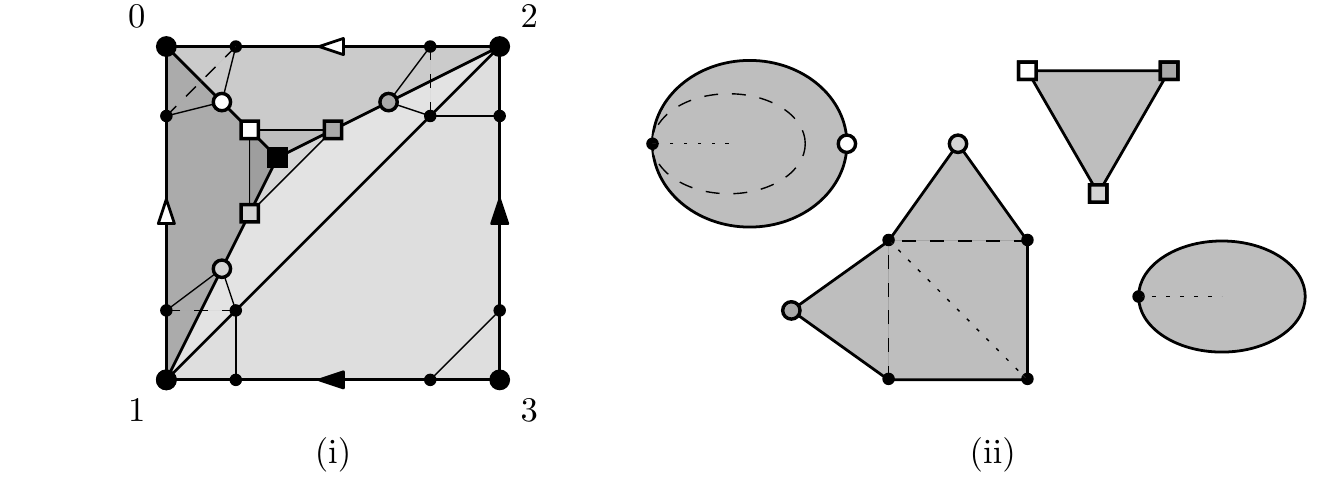}
    \caption{The complex $\tri'_2$ (i), and its four vertex links (ii).}
    \label{fig:quadlk}
  \end{figure} 

\begin{figure}[htb]
	\centering
	\includegraphics[scale=1]{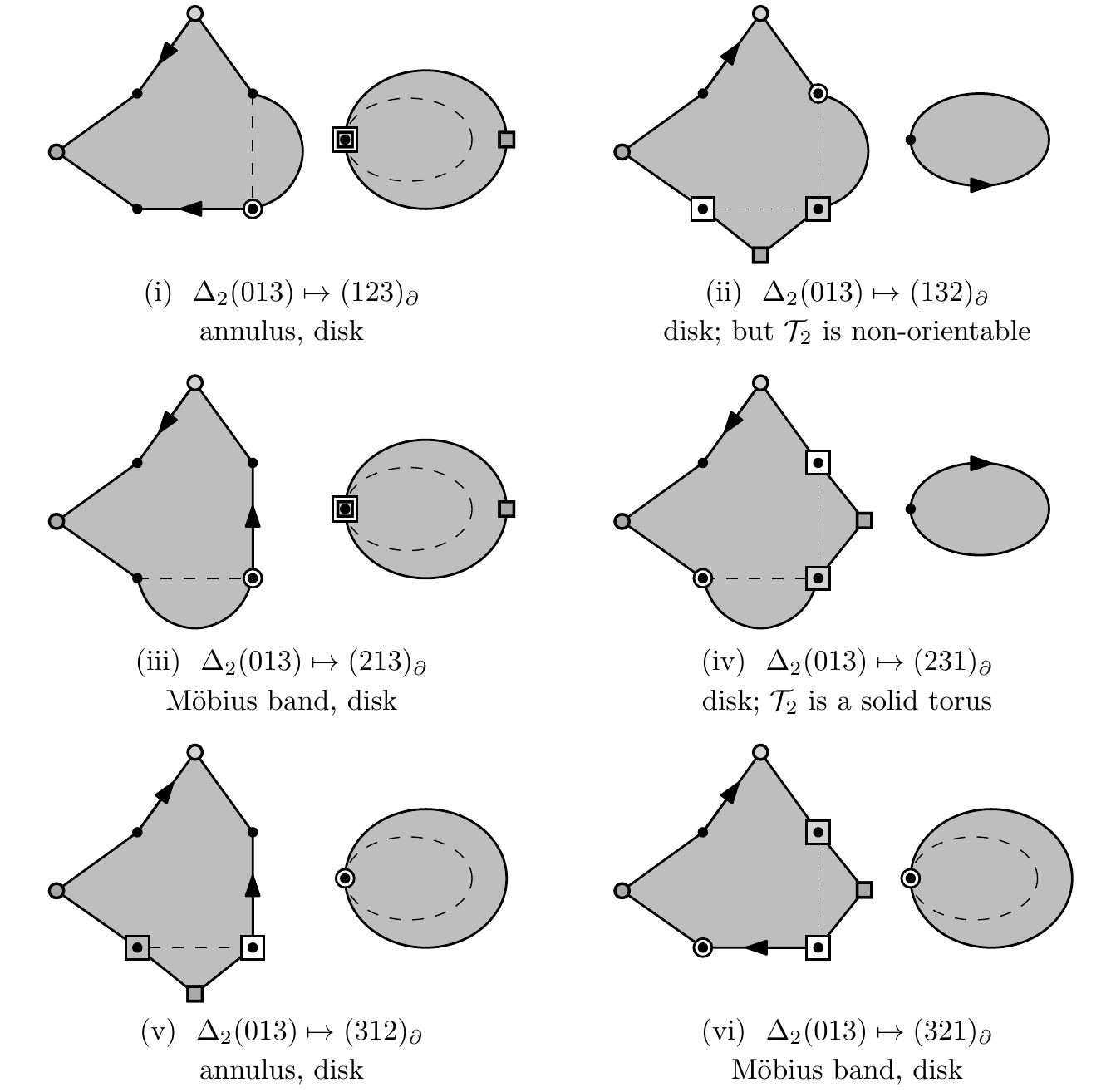}
    \caption{}
    \label{fig:sp3b_cases_1}
\end{figure}

\begin{description}
	\item[$ \Delta_2 (013) \mapsto (123)_\partial$] In this case we can see that $\tri_2$ has as boundary two triangles identified along their boundaries with two vertices identified. The two vertex links are an annulus and a disk, and all three boundary components of the links are bigons.
	
	Extending this case to a valid complex $\tri_3$ is only possible in the trivial way, i.e., gluing $\Delta_3$ to $\tri_2$ along $\Delta_3(123) \mapsto (123)_{\partial}$ and $\Delta_3(023) \mapsto (123)_{\partial}$. This yields a $2$-vertex complex with boundary isomorphic to that of the snapped 3-ball but with one apex identified with the vertex of the loop edge. Again, we have one annulus and one disk as vertex links. In accordance with the structure of $\partial \tri_3$, the disk is now bounded by a single normal arc while the annulus has a loop normal arc as one boundary component, and four normal arcs in the other.
	
	We can glue $\Delta_4$ to the unique valid complex $\tri_3$ described in the previous paragraph in twelve distinct ways: We start by $\Delta_4(012) \mapsto (012)_{\partial}$ and proceed by gluing $\Delta_4(013)$ and $\Delta_4(023)$ to $(012)_{\partial}$ in all possible six ways each.

	Apart from the trivial gluing, which results in the same type as $\tri_2$ above, we obtain three 1-vertex complexes with boundary a 2-sphere with vertices identified and vertex link a 3-punctured sphere. Two of them have a vertex link with three boundary components consisting of two edges each (i.e., $\partial \tri_4$ is of type two triangles identified along their boundaries and all vertices identified), one of them has a vertex link with two boundary components with one and one boundary component with four edges (i.e., the boundary $\partial \tri_4$ is isomorphic to that of the snapped 3-ball with all vertices identified). These cases correspond to types $\mathscr{T}_{\text{III}}$ and $\mathscr{T}_{\text{II}}$ respectively, from the proof of \Cref{lem:lst}.

	\item[$ \Delta_2 (013) \mapsto (132)_\partial$] yields a non-orientable 1-handle with Klein bottle boundary. This cannot be completed to an orientable 3-manifold and thus we are done with this case.

	\item[$ \Delta_2 (013) \mapsto (213)_\partial$] is invalid, as it produces a M\"obius band in one of the vertex links. 

	\item[$ \Delta_2 (013) \mapsto (231)_\partial$] gives a 1-vertex triangulation of the solid torus. In particular, the vertex link is a triangulated hexagon with neighboring edges in the boundary of the link being normal arcs in triangles of $\partial \tri_2$. We can thus proceed as in the proof of \Cref{lem:lst} to conclude that $\tri$ must be of Heegaard genus one.

	\item[$ \Delta_2 (013) \mapsto (312)_\partial$] produces a 2-sphere of type ``boundary of the snapped $3$-ball'' in the boundary, and two vertex links. One of them an annulus, the other one a disk. Here, the disk is bounded by a single normal arc while the annulus has a loop normal arc as one boundary component, and four normal arcs in this other. This case is of the same type as $\tri_3$ in case $ \Delta_2 (013) \mapsto (123)_\partial$ above.

	\item[$ \Delta_2 (013) \mapsto (321)_\partial$] creates a M\"obius band in the vertex link and can thus be discarded. 

\begin{figure}[ht]
	\centering
	\includegraphics[scale=1]{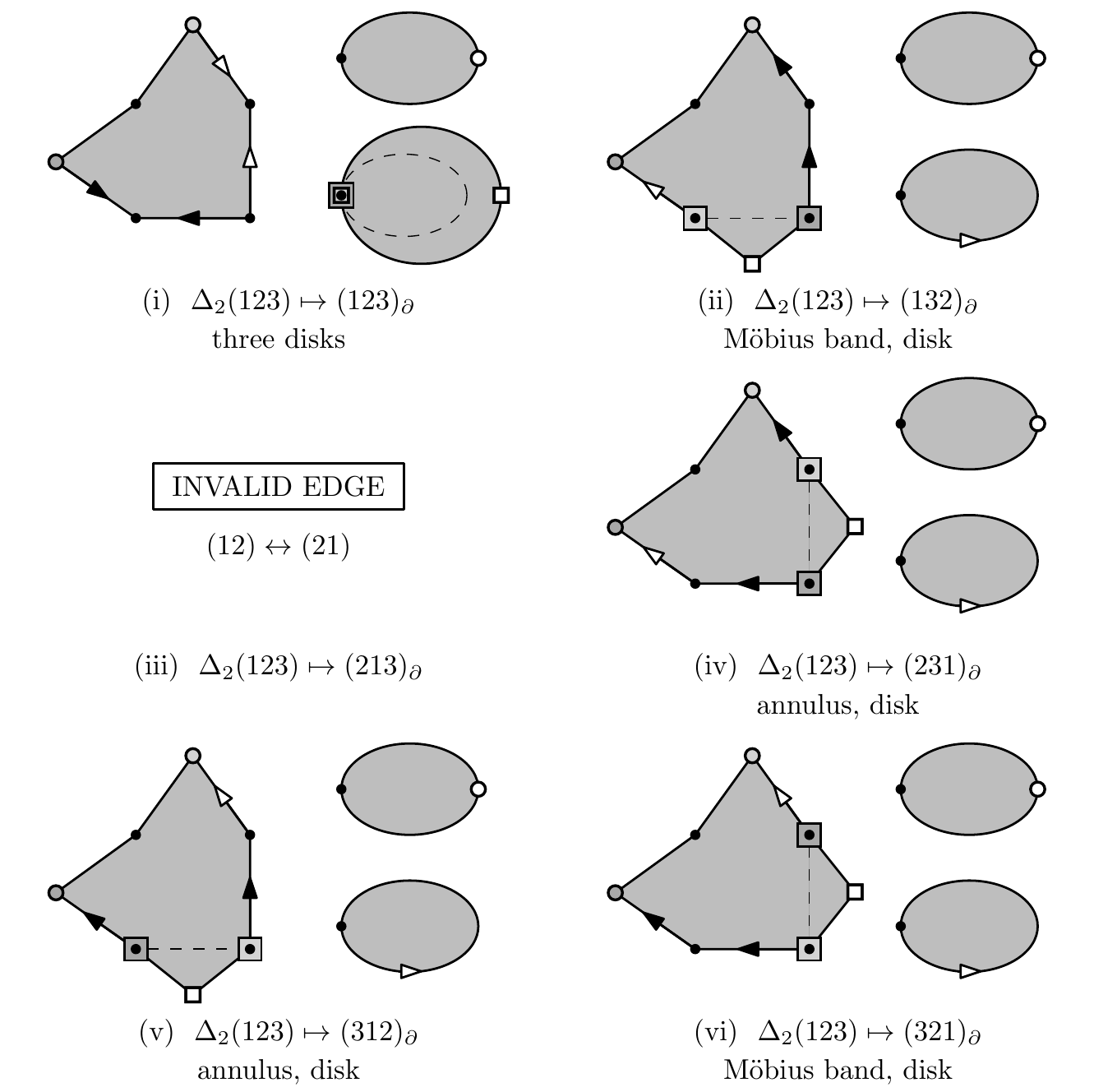}
    \caption{}
    \label{fig:sp3b_cases_2}
  \end{figure}

	\item[$ \Delta_2 (123) \mapsto (123)_\partial$] is a layering that creates an interior degree two edge. Consequently we obtain two triangles glued along their boundaries for $\partial \tri_3$. The three vertex links are all disks with two boundary edges. 

	Extending this complex by attaching $\tri_3$ yields three valid complexes. The first is obtained via the trivial gluing and of the same type as the snapped 3-ball $\tri_1$. The other two are 1-vertex triangulations of the solid torus already discussed in \Cref{lem:lst}.

	\item[$ \Delta_2 (123) \mapsto (132)_\partial$] yields a M\"obius band in the vertex link and can be discarded. 

	\item[$ \Delta_2 (123) \mapsto (213)_\partial$] identifies $(12)$ on the boundary with itself in reverse, thus is invalid.

	\item[$ \Delta_2 (123) \mapsto (231)_\partial$] This gluing, again, produces two vertex links. One of them an annulus, the other one a disk. The boundary of $\tri_2$ and the boundary components of the vertex links are of the same type as in the case $ \Delta_2 (013) \mapsto (123)_\partial$. 

	\item[$ \Delta_2 (123) \mapsto (312)_\partial$] gives an annulus and a disk as vertex links. $\partial\tri_2$ and the boundary components of the vertex links are of the same type as in the case $ \Delta_2 (013) \mapsto (123)_\partial$. 

	\item[$ \Delta_2 (123) \mapsto (321)_\partial$] produces a M\"obius band in the vertex link and can be discarded. 
\end{description}

It remains to show---along the same lines as in \Cref{lem:lst}---that none of the complexes described above can be completed to a triangulation of Heegaard genus greater than one. It suffices to look at the complexes which can be completed to a triangulation of a manifold. 

The 1-vertex complexes with torus boundary ($ \Delta_2 (013) \mapsto (231)_\partial$ and $ \Delta_2 (123) \mapsto (123)_\partial$) are solid tori and thus admit a fundamental group with one generator. Following the proof of \Cref{lem:lst}, the Heegaard genus of a triangulation of a closed 3-manifold obtained from these subcomplexes are of Heegaard genus at most one.

The three 1-vertex complexes with vertex links a 3-punctured sphere have as fundamental group the free group with two generators. However, note that these complexes can only be extended by trivial gluings and completed by inserting a 1-tetrahedron snapped 3-ball (or, in the case of three boundary components of size two, closed off by trivially identifying the two boundary triangles. In all of these cases we obtain a triangulation of the 3-sphere and in particular a closed 3-manifold of Heegaard genus at most one.
\end{proof}

\section{3-Manifolds of treewidth two}
\label{sec:tw_two}

In what follows, we use the classification of $3$-manifolds of treewidth one (\Cref{thm:main}) to show that a large class of orientable Seifert fibered spaces and some {\em graph manifolds} have treewidth two. This is done by exhibiting appropriate triangulations, which have all the hallmarks of a space station. First, we give an overview of the building blocks.

\paragraph*{Robotic arms.} These are the layered triangulations of the solid torus with 2-triangle boundaries introduced in \Cref{ssec:layered} and encountered in the proof of \Cref{thm:main}. Their dual graphs are thick paths (\Cref{fig:meridian}(v)). A layered solid torus is of type $\lst (p,q,r)$ if its meridional disk intersects its boundary edges $p$, $q$ and $r$ times. For any coprime $p,q,r$ with $p+q+r = 0$, a triangulation of type $\lst (p,q,r)$ can be realized by \cite[Algorithm 1.2.17]{burton2003minimal}.

\begin{example}
\label{ex:lst01}
A special class of robotic arms are the ones of type $\lst (0,1,1)$, as they can be used to trivially fill-in superfluous torus boundary components without inserting an unwanted exceptional fiber into a Seifert fibered space (cf.\ descriptions of $\disk_2$ and $\disk_1$ below).
One of the standard triangulations of robotic arms of type $\lst (0,1,1)$ has three tetrahedra $\Delta_i$, $0 \leq i \leq 2$, and is given by the gluing relations \eqref{eq:lst01}.

\begin{align}
  	\begin{split}
        \Delta_0 (023) \mapsto \Delta_1 (013), \quad \Delta_0 (123) \mapsto \Delta_1 (120), \quad \Delta_1 (023) \mapsto \Delta_2 (201), \\
        \Delta_1 (123) \mapsto \Delta_2 (301), \quad \Delta_2 (023) \mapsto \Delta_2 (312).
    \end{split}
    \label{eq:lst01}  
\end{align}
\end{example}

\paragraph*{Core unit with three docking sites.} Start with a triangle $t$, take the product $t \times [0,1]$, triangulate it using three tetrahedra, \Cref{fig:base}(i)--(ii), and glue $t \times \{0\}$ to $t \times \{1\}$ without a twist, \Cref{fig:base}(iii). The dual graph of the resulting complex $\disk_3$---topologically a solid torus---is $K_3$, hence of treewidth two. Its boundary---a $6$-triangle triangulation of the torus---can be split into three 2-triangle annuli, corresponding to the edges of $t$, each of which we call a {\em docking site}. Edges running along a fiber and thus of type $\{\text{vertex~of}~t\}\times [0,1]$ are termed {\em vertical edges}. Edges orthogonal to the fibers, i.e., the edges of $t\times \{0\} = t \times \{1\}$, are termed {\em horizontal edges}. The remaining edges are referred to as {\em diagonal edges}.

More concisely, the triangulation of $\disk_3$ has gluing relations
\begin{align}
  \label{eq:baseA}
  \Delta_0 (012) \mapsto \Delta_1 (012), \quad \Delta_1 (013) \mapsto \Delta_2 (013), \quad \Delta_2 (023) \mapsto \Delta_0 (312). 
\end{align}

\begin{figure}[ht]
	\centering
	\includegraphics[scale=1]{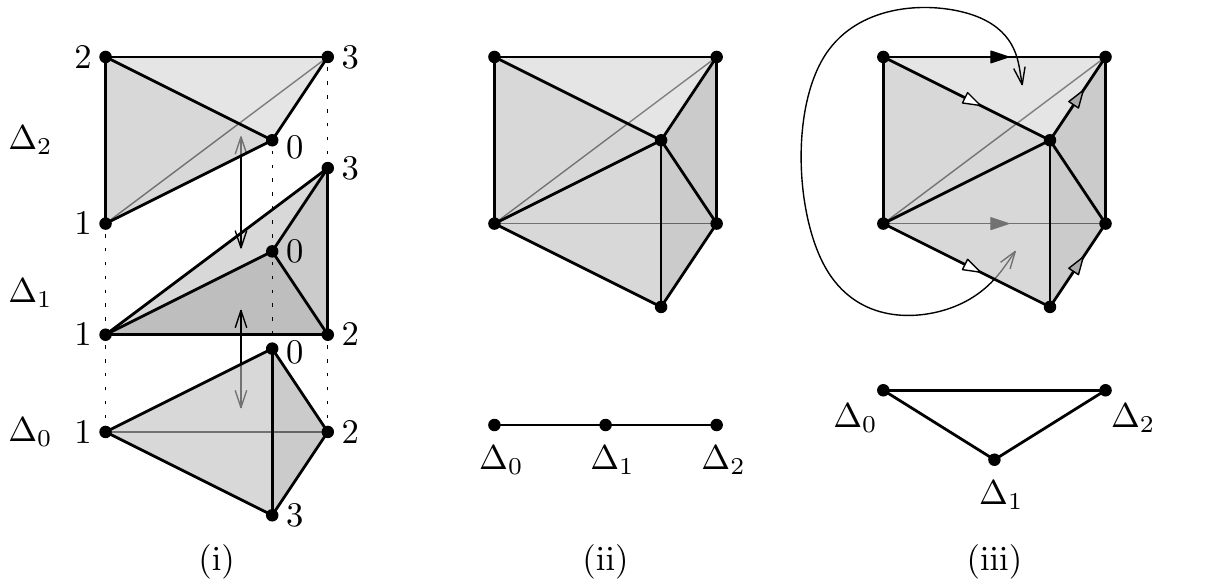}
    \caption{Construction of the core unit $\disk_3$ with three docking sites.\label{fig:base}}
\end{figure}

\paragraph*{Core assembly with $r$ docking sites.}

For $r=2$ (resp.\ $r=1$), take a core unit $\disk_3$ and glue a robotic arm of type $\lst (0,1,1)$ onto one (resp.\ two) of its docking sites such that the {\em unique boundary edge} of the robotic arm (i.e., the boundary edge which is only contained in one tetrahedron of the layered solid torus) is glued to a horizontal boundary edge of $\disk_3$ (see \Cref{ex:lst01} for a detailed description of a particular triangulation of a layered solid torus of type $\lst (0,1,1)$ with unique boundary edge being $\Delta_0 (01)$). The resulting complex is denoted by $\disk_2$ (resp.\ $\disk_1$) and their dual graphs are shown in \Cref{fig:base}. Observe that they have treewidth two.

\begin{figure}[ht]
	\centering
	\includegraphics[scale=1]{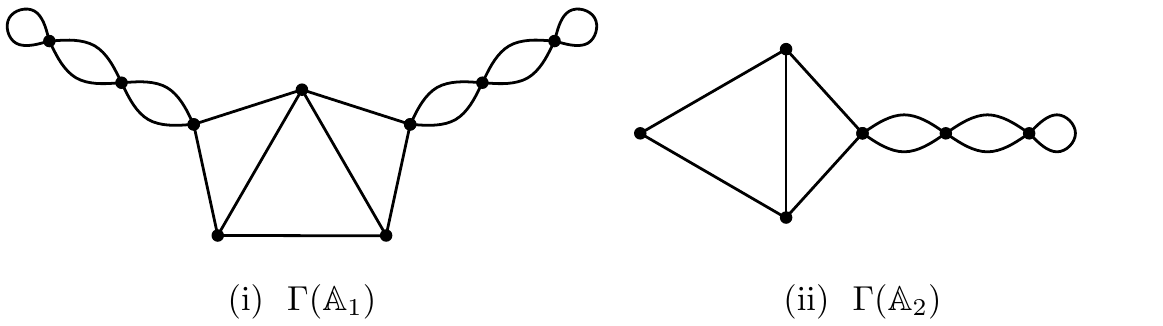}
    \caption{\label{fig:a2a1}}
\end{figure}

For $\disk_r$ ($r\geq 3$) take $r-2$ copies of $\disk_3$, denote them by $\disk_3^i$, $1 \leq i \leq r-2$, with tetrahedra $\Delta_0^i$, $\Delta_1^i$ and $\Delta_2^i$, $1 \leq i \leq r-2$. Glue them together by mirroring them across one of their docking sites as shown by \Cref{eq:base_odd} for $1 \leq i \leq r-3$ odd, and by \Cref{eq:base_even} for $2 \leq i \leq r-3$ even. 
The resulting complex, denoted by $\disk_r$ (see \Cref{fig:nbase} for $r=5$), has $2r$ boundary triangles which become $r$ docking sites.
\begin{align}
  \label{eq:base_odd}
  i~\text{odd:}~~&\Delta_1^i (123) \mapsto \Delta_1^{i+1} (123),\quad  \Delta_2^i (123) \mapsto \Delta_2^{i+1} (123). \\
  \label{eq:base_even}
  i~\text{even:}~~&\Delta_0^i (023) \mapsto \Delta_0^{i+1} (023),\quad  \Delta_1^i (023) \mapsto \Delta_1^{i+1} (023). 
\end{align}

\begin{figure}[ht!]
	\centering
	\includegraphics[scale=1]{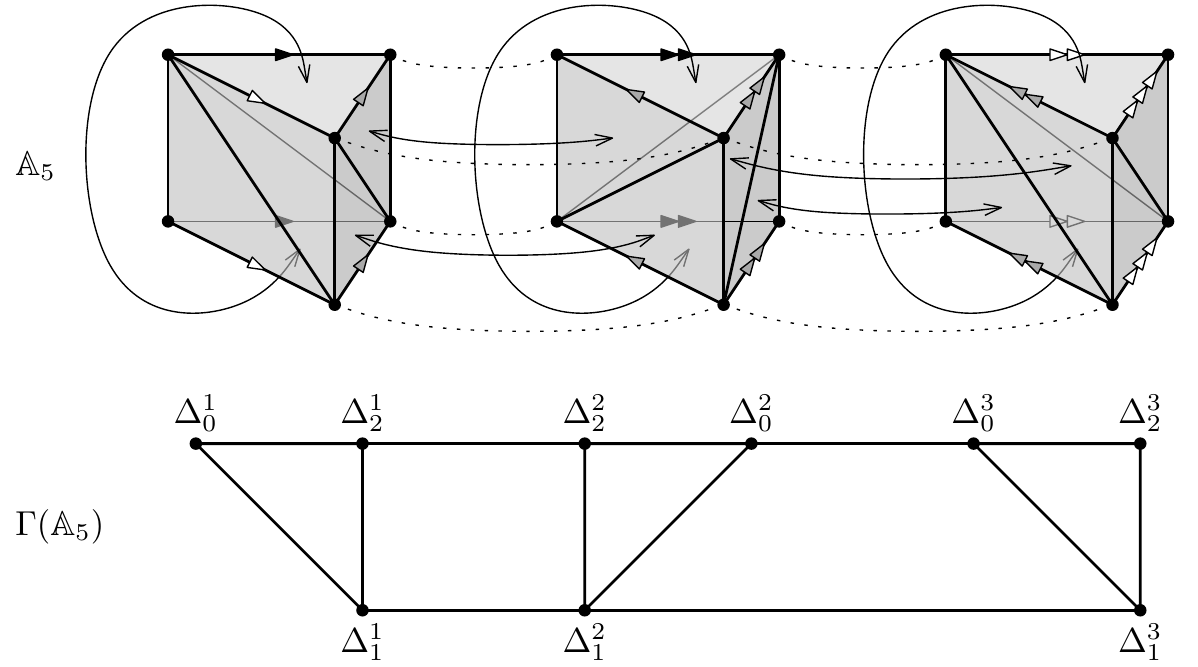}
  	\caption{The core assembly with 5 docking sites and its dual graph.\label{fig:nbase}}
\end{figure}

\paragraph*{M\"obius laboratory module.} This complex, denoted by $\moebius$, is given by 
\begin{align}
\label{eq:baseM}
\begin{split}
  T_0 (123) \mapsto T_1 (123),\quad  T_0 (023) \mapsto T_1 (031),\quad  T_1 (012) \mapsto T_2 (201)\\
  T_1 (023) \mapsto T_2 (023),\quad  T_0 (013) \mapsto T_2 (132).
\end{split}
\end{align}
Its dual graph is a triangle with two double edges, and hence of treewidth two (see, for instance, \Cref{fig:nonorsfs}). $\moebius$ has one torus boundary component, or docking site, given by the two triangles $T_0(012)$ and $T_2(013)$ with edges $T_0(01) = T_2(13)$, $T_0(02)=T_2(03)$, and $T_0(12) = T_2(01)$. $\moebius$ triangulates the orientable $\nsphere{1}$-bundle over $\nsurface_{1,1}$.

\begin{theorem}
\label{thm:sfs1}
Orientable Seifert fibered spaces over $\nsphere{2}$ have treewidth at most two.
\end{theorem}

\begin{proof}
To obtain a treewidth two triangulation of $\sfs [\nsphere{2} : (a_1, b_1), \ldots , (a_r, b_r)]$, start with the core assembly $\disk_r$ and a collection of robotic arms $\lst(a_i,\pm |b_i|, -a \mp |b_i|)$, $1 \leq i \leq r$. The robotic arms are then glued to the $r$ docking sites ($2$-triangle torus boundary components, separated by the vertical boundary edges) of $\disk_r$, such that boundary edges of type $a_i$ are glued to vertical boundary edges, and edges of type $b_i$ are glued to horizontal boundary edges. The sign in $\lst(a_i,\pm|b_i|, -a \mp |b_i|)$ is then determined by the type of diagonal edge in the $i$\textsuperscript{th} docking site of $\disk_r$ and by the sign of $b_i$. (See \Cref{fig:spheresfs} for an example.)
\end{proof}

\begin{figure}[ht]
	\centering
	\includegraphics[scale=1]{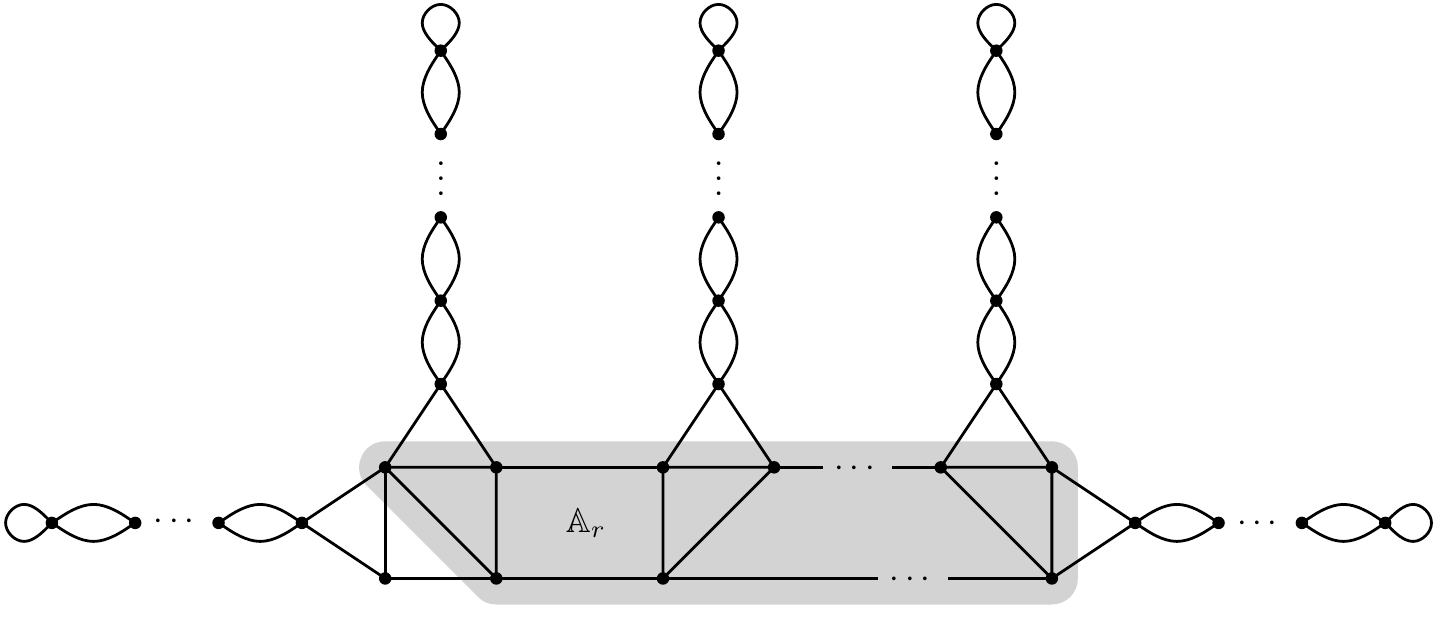}
    \caption{Dual graph of a treewidth two triangulation of an orientable SFS over $\nsphere{2}$. \label{fig:spheresfs}}
\end{figure}

\begin{remark}
\label{rem:onetwo}
Note that, in some cases, a fibre of type $(2,1)$ can be realized by directly identifying the two triangles of a docking site of $\disk_r$ with a twist. In the dual graph this appears as a double edge rather than the attachment of a thick path. See \Cref{fig:sigma3} on the right for an example in the treewidth two triangulation of the Poincar\'e homology sphere.
\end{remark}

\begin{theorem}
\label{thm:sfs2}
An orientable SFS over a non-orientable surface is of treewidth at most two.
\end{theorem}

\begin{figure}[ht!]
	\centering
	\includegraphics[scale=1]{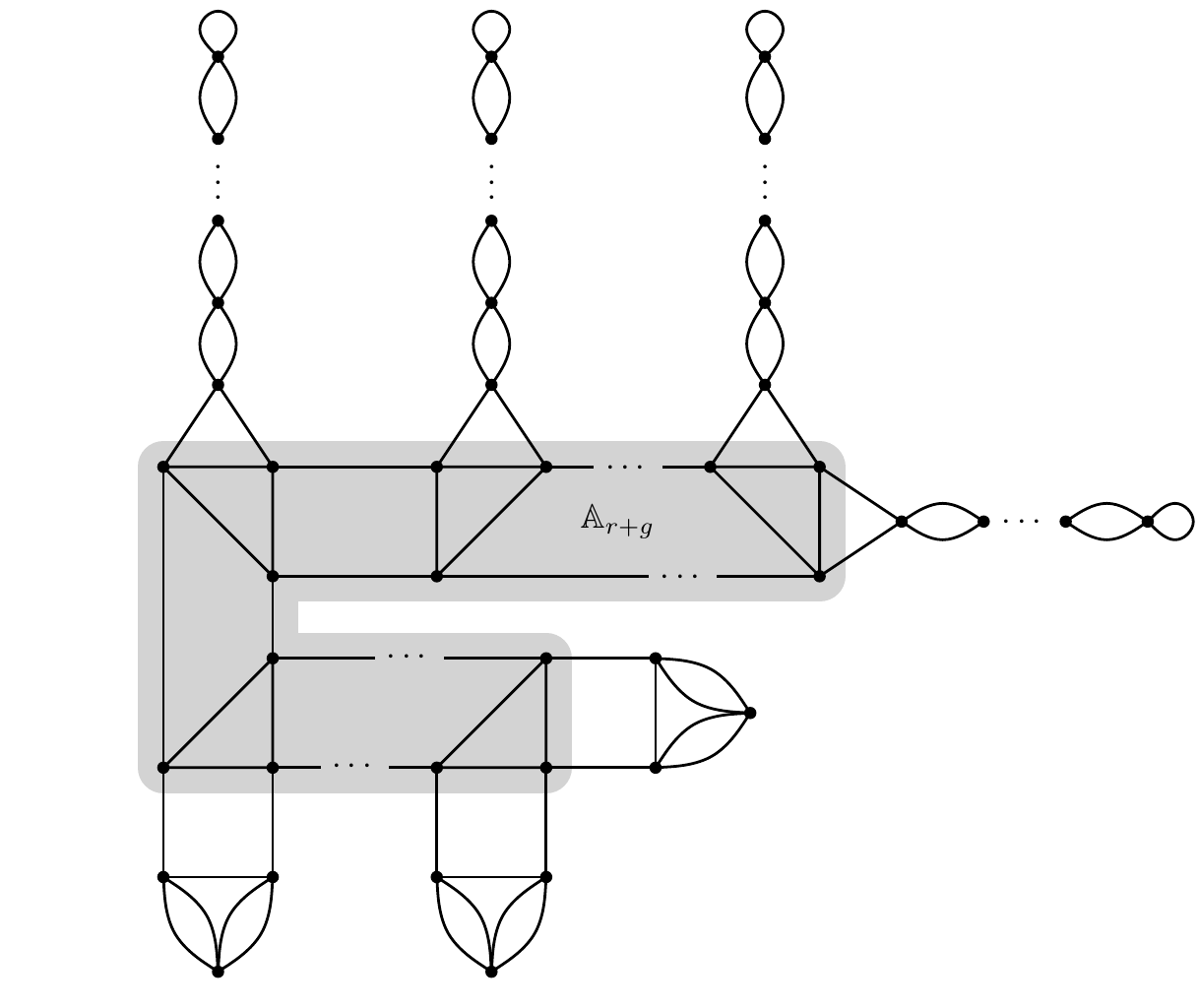}
    \caption{Dual graph of a treewidth two triangulation of an orientable SFS over $\nsurface_g$. \label{fig:nonorsfs}}
  \end{figure}

\begin{proof}
In order to obtain a treewidth two triangulation of the orientable Seifert fibered space $\sfs[\nsurface_{g} : (a_1, b_1), \ldots , (a_r, b_r)]$ over the non-orientable surface $\nsurface_{g}$ of genus $g$, start with a core assembly $\disk_{r+g}$ and attach $g$ copies $\moebius_j$, $1 \leq j \leq g$, of the M\"obius laboratory module via
\begin{align}
	\label{eq:nbase_even}
	T_0^j (012) \mapsto \Delta_2^{j} (201) \quad \text{and} \quad T_2^j (013) \mapsto \Delta_0^{j} (013), 
\end{align}
where $T_0^j$, $T_1^j$ and $T_2^j$ are the tetrahedra comprising $\moebius_j$, and $\Delta_0^j$, $\Delta_1^j$ and $\Delta_2^j$ denote the tetrahedra making up the first $g$ core units (each being a copy of $\disk_3$) in $\disk_{r+g}$. By construction, this produces a triangulation of the orientable $\nsphere{1}$-bundle over $\nsurface_{g,1}$.  
Proceed by attaching a robotic arm of type $\lst(a_i,\pm|b_i|,-a_i \mp |b_i|)$, $1 \leq i \leq r$, to each of the remaining $r$ docking sites. Again, for the gluings between the robotic arms and the core assembly $\disk_{r+g}$, the edges of type $a_i$ must be glued to the vertical boundary edges, the edges of type $b_i$ must be glued to the horizontal boundary edges, and attention has to be paid to the signs of the $b_i$ and to how exactly diagonal edges run. See \Cref{fig:nonorsfs} for a picture of the dual graph of the resulting complex, which is of treewidth two by inspection.
\end{proof}

\begin{cor}
An orientable Seifert fibered space $\manifold$ over $\nsphere{2}$ or a non-orientable surface is of treewidth one, if $\manifold$ is a lens space or $\sfs[\nsphere{2} : (2,1),(2,1),(2,-1)]$, and two otherwise.
\end{cor}

This statement directly follows from the combination of \Cref{thm:main,thm:sfs1,thm:sfs2}.

\begin{cor}
\label{cor:spherical}
Orientable spherical or \emph{``$\nsphere{2} \times \mathbb{R}$''} $3$-manifolds are of treewidth at most two. 
\end{cor}

\begin{proof}
Every orientable 3-manifold with spherical or $\nsphere{2} \times \mathbb{R}$ geometry can be represented either as a Seifert fibered space over the 2-sphere with at most three exceptional fibers, or as a Seifert fibered space over the projective plane (i.e., $\nsurface_1$) with at most one exceptional fiber. Hence the result follows directly from \Cref{thm:sfs1,thm:sfs2}.
\end{proof}

\begin{cor}
\label{cor:graph}
Graph manifolds $\manifold_T$ modeled on a tree $T$ with nodes being orientable Seifert fibered spaces over $\nsphere{2}$ or $\nsurface_{g}$, $g>0$, have treewidth at most two.
\end{cor}

\begin{proof}[Proof of \Cref{cor:graph}]
Let $\manifold$ be such a graph manifold. A treewidth two triangulation of $\manifold$ can be constructed in the following way: For every node of $T$ of degree $k$ insert a treewidth two triangulation of the respective Seifert fibered space from \Cref{thm:sfs1} or \Cref{thm:sfs2} with $k$ additional docking sites. As can be deduced from the construction of $\disk_r$, this can be done without increasing the treewidth. Proceed by gluing the Seifert fibered spaces along the arcs of $T$ using the additional docking sites (torus boundary components) added in the previous step: Every such gluing is determined by a diffeomorphism on the torus and every such diffeomorphism can be modelled by a sequence of layerings onto the boundary components with dual graph a path of double edges of treewidth one. This fact can be deduced from, for instance, \cite[Lemma 4.1]{Jaco06Layering}. Altogether, the triangulation constructed this way is of treewidth at most two. See \Cref{fig:graphmfd} for an example.
\end{proof}

\begin{cor}
Minimal triangulations are not always of minimum treewidth.
\label{cor:minimal}
\end{cor}

\begin{proof}
The {\em Poincar\'e homology sphere} $\Sigma^3 = \sfs [\nsphere{2} : (2,1),(3,1),(5,-4) ]$ has treewidth two but its minimal triangulation has treewidth four, see \Cref{fig:sigma3}.
\end{proof}

\begin{figure}[ht]
	\centering
	\includegraphics[width=12cm]{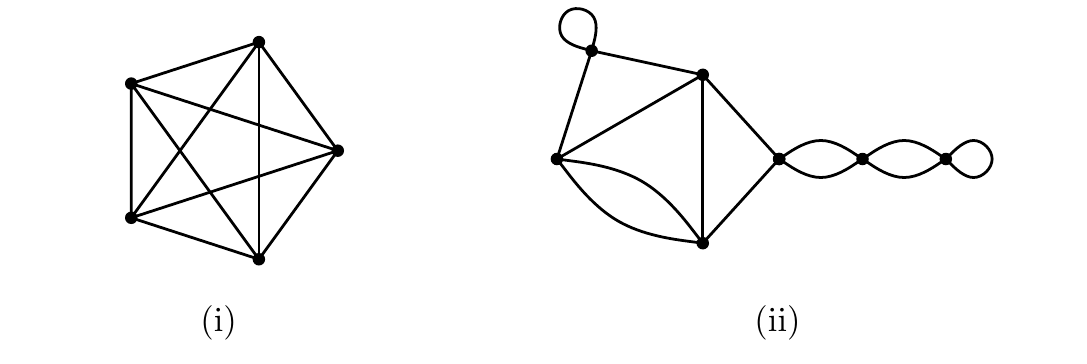}
    \caption{Dual graph of the minimal (i), and of a treewidth two (ii) triangulation of $\Sigma^3$, see also \Cref{rem:onetwo}.\label{fig:sigma3}}
\end{figure}

\begin{cor}
\label{cor:lowtwhighhg}
There exist irreducible $3$-manifolds with treewidth two, but arbitrarily high Heegaard genus.
\end{cor}

\begin{proof}
Due to work of Boileau and Zieschang \cite[Theorem 1.1 (i)]{boileau1984heegaard}, the 3-manifold $\manifold_m = \sfs [\nsphere{2} : (2,1), \ldots ,(2,1), (a_m,b_m) \ ]$, for $a_m$ odd, has Heegaard genus $\mathfrak{g}(\manifold_m) \geq m-2$. At the same time, $\tw{\manifold_m}=2$ due to \Cref{thm:sfs2}.
\end{proof}

\newpage

By combining \Cref{thm:main,thm:sfs1,thm:sfs2}, we can determine the treewidth of most closed, orientable $3$-manifolds which admit a triangulation with $\leq 10$ tetrahedra, see \Cref{tab:tw}.

\begin{figure}[h!]
    {\centering
    \begin{tabular}{rr|rrrr}
      \toprule
      $n$  & $\#$ mfds. $\manifold$ &$\tw{\manifold}=0$&$\tw{\manifold}=1$&$\tw{\manifold}=2$& unknown  \\
      \midrule
      $ 1$     & $   3$ & $3$ & $   0$ & $   0$ & $  0$ \\
      $ 2$     & $   7$ & $0$ & $   7$ & $   0$ & $  0$ \\
      $ 3$     & $   7$ & $0$ & $   6$ & $   1$ & $  0$ \\
      $ 4$     & $  14$ & $0$ & $  10$ & $   4$ & $  0$ \\
      $ 5$     & $  31$ & $0$ & $  20$ & $  11$ & $  0$ \\
      $ 6$     & $  74$ & $0$ & $  36$ & $  36$ & $  2$ \\
      $ 7$     & $ 175$ & $0$ & $  72$ & $ 100$ & $  3$ \\
      $ 8$     & $ 436$ & $0$ & $ 136$ & $ 297$ & $  3$ \\
      $ 9$     & $1154$ & $0$ & $ 272$ & $ 861$ & $ 21$ \\
      $10$     & $3078$ & $0$ & $ 528$ & $2489$ & $ 61$ \\
      \midrule
      $\Sigma$ & $4979$ & $3$ & $1087$ & $3799$ & $ 90$ \\
      \bottomrule
    \end{tabular}
    \caption{The $4979$ $3$-manifolds triangulable with $\leq 10$ tetrahedra and their treewidths.\label{tab:tw}}}
\end{figure}

\begin{remark}
  Using similar constructions it can be shown that orientable Seifert fibered spaces over orientable surfaces have treewidth at most four. Naturally, this only provides an upper bound rather than the actual treewidth of this family of $3$-manifolds. Determining the maximum treewidth of an orientable Seifert fibered space is thus left as future work.
\end{remark}

\newpage

\begin{figure}[h!]
	\centering
	\includegraphics[scale=1]{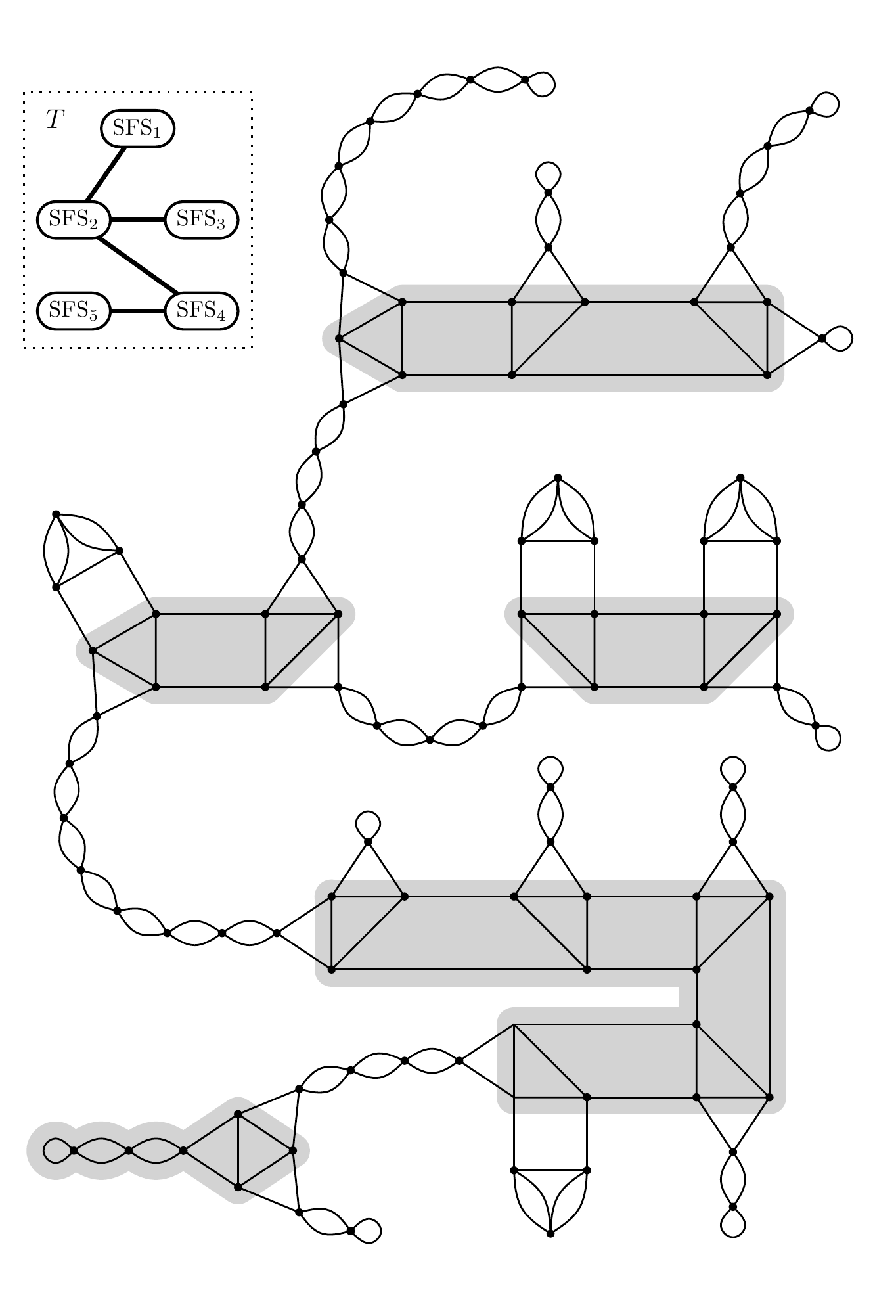}
	\caption{Dual graph of a treewidth two triangulation of a graph manifold modeled on the tree $T$. In order to increase visibility, the core of each constituent Seifert fibered space is highlighted in gray. \label{fig:graphmfd}}
\end{figure}

\newpage


\bibliographystyle{my_plainurl}
\bibliography{references}

\newpage

\appendix

\section{The 1-tetrahedron layered solid torus}
\label{app:torus}

\begin{figure}[!ht]
	\centering
	\includegraphics[scale=1]{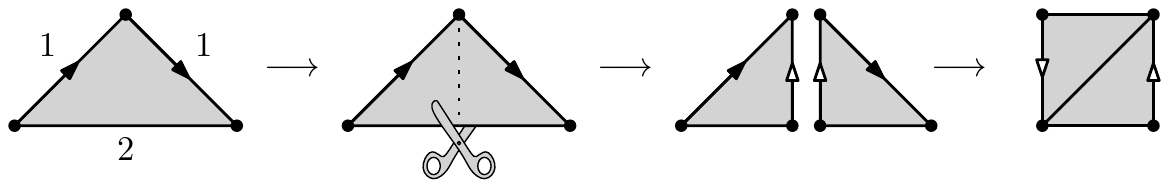}
	\caption{A 1-triangle M\"obius band in various disguises.}
	\label{fig:one_triangle_moebius}
\end{figure}

\begin{figure}[!ht]
	\centering
	\includegraphics[scale=0.5]{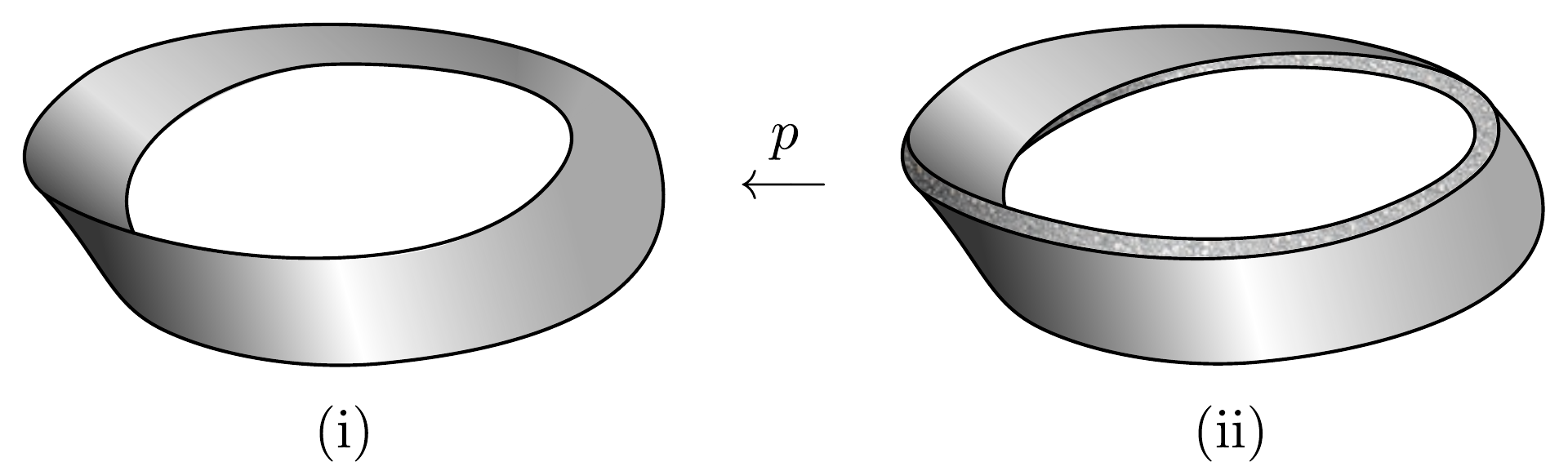}
	\caption{(i) A M\"obius band $M$, and (ii) its ``thickening'' $\boldsymbol{M}$, which can be described as an $I$-bundle over $M$. The fiber-wise boundary of $\boldsymbol{M}$ is homeomorphic to an annulus and covers $M$ twice under the projection $p \colon\boldsymbol{M}\rightarrow M$ onto the base space $M$.}
	\label{fig:thickened_moebius}
\end{figure}

\begin{figure}[!ht]
	\centering
	\includegraphics[scale=1]{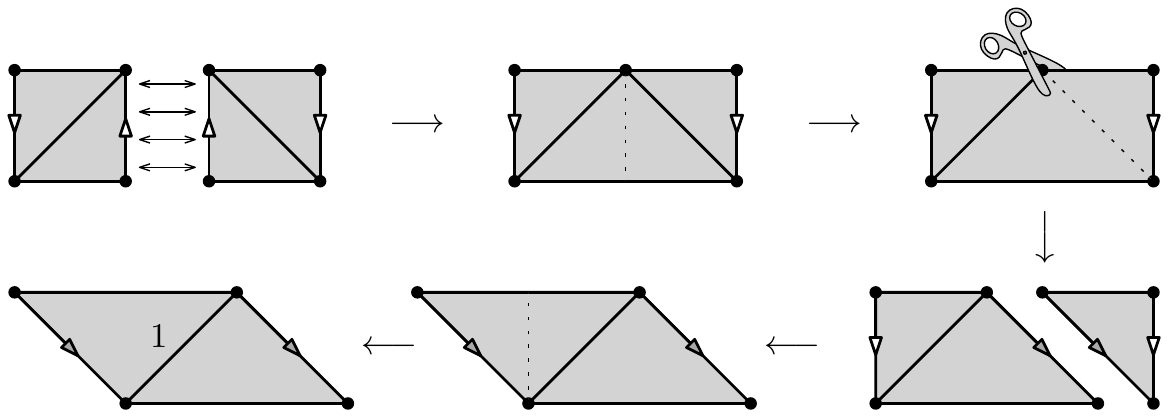}
	\caption{Constructing the orientable 2-cover of the M\"obius band, which is an annulus. After layering a tetrahedron on edge 1, we may compose this layering with the projection of the annulus onto the M\"obius band to obtain a solid torus.}
	\label{fig:orientable_2-cover}
\end{figure}

\section{High-treewidth triangulations of 3-manifolds}
\label{app:high_tw}

\begin{prop}
\label{prop:arbitrary_high}
Every 3-manifold admits a triangulation of arbitrarily high treewidth.
\end{prop}

\begin{proof}[Sketch of the proof.]
Since the treewidth is monotone with respect to taking subgraphs \cite[Lemma 11 (Scheffler)]{bodlaender1998partial}, it is sufficient to exhibit high-treewidth triangulations for the 3-ball, which then can be connected (via the `connected sum' operation) to any triangulation.

For every $k \in \mathbb{N}$, however, it is easy to construct a triangulation of the 3-ball, whose dual graph contains the $k \times k$-grid as a minor (\Cref{fig:grid,fig:grid_dual}). Since the treewidth is minor-monotone \cite[Lemma 16]{bodlaender1998partial} and $\tw{k \times k\text{-grid}}=k$ \cite[Section 13.2]{bodlaender1998partial}, the result follows.
\end{proof}

\begin{remark} There is another approach  \cite{bagchi2016efficient}, making use of the existence of arbitrarily large {\em simplicial $2$-neighborly triangulations} of 3-manifolds (cf.\ \cite{sarkaria1983neighborly}, and \cite[Section 7]{walkup1970lower}).
\end{remark}

\begin{figure}[ht!]
	\centering
	\includegraphics[scale=1]{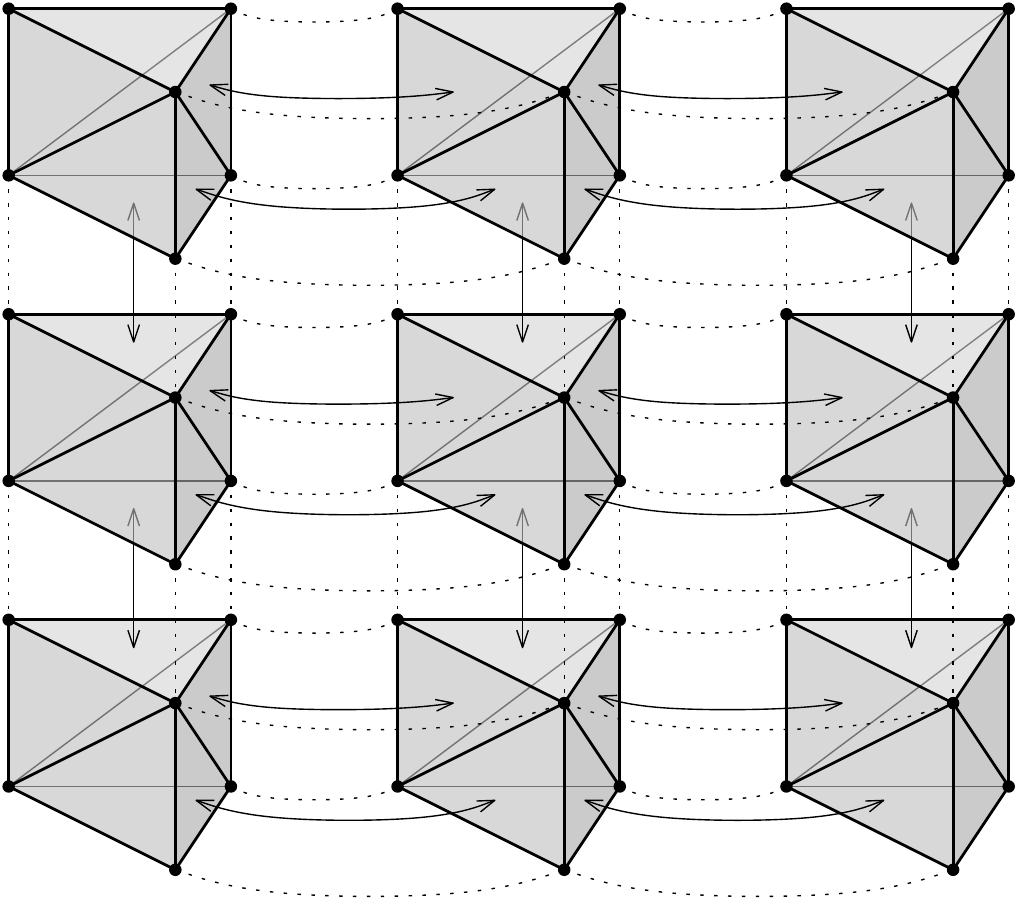}
	\caption{A ``grid-like'' triangulation $\tri_{3 \times 3}$ of the 3-ball.\label{fig:grid}}
\end{figure}

\bigskip

\begin{figure}[ht!]
	\centering
	\includegraphics[scale=1]{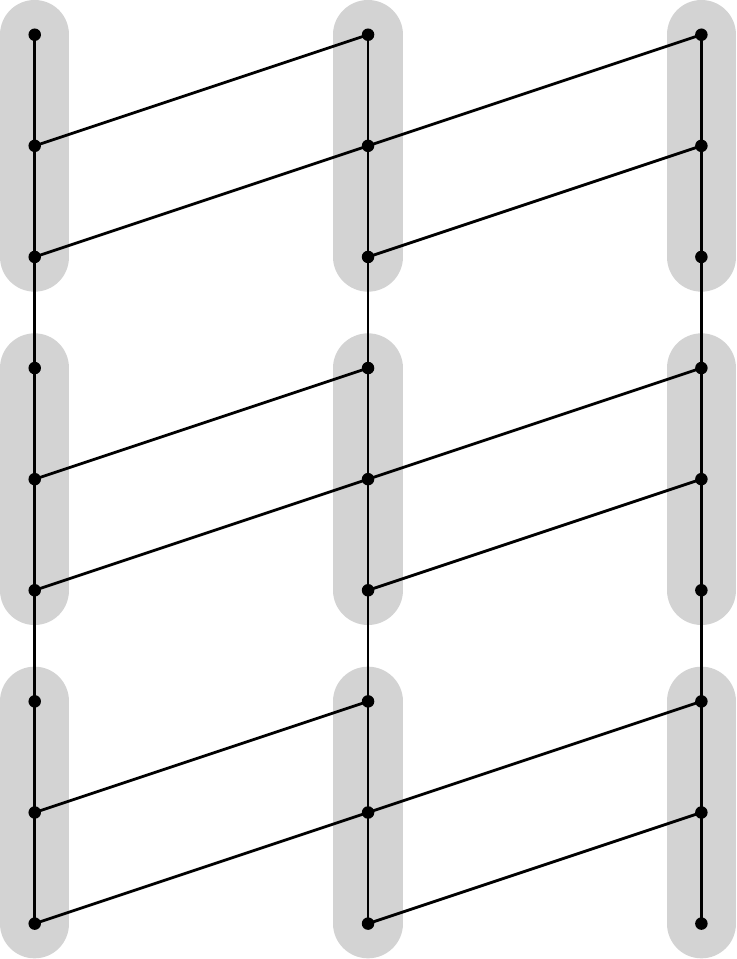}
	\caption{The dual graph of $\tri_{3 \times 3}$ from \Cref{fig:grid} containing the $3\times 3$ grid as a minor. \label{fig:grid_dual}}
\end{figure}

\newpage

\section{An algorithmic aspect of layered triangulations}
\label{app:application}

As an application of \Cref{thm:cutwidth}, we describe how the machinery of layered triangulations together with work of Bell \cite{Bell15Diss} can be employed to construct a ``convenient'' triangulation of a 3-manifold when it is presented via a mapping class $w \in \MCG(\altsurface_g^\star)$.\footnote{That is, $\manifold = \hbody \cup_f \hbody'$ where $\hbody$ and $\hbody'$ are genus $g$ handlebodies and $f$ belongs to the isotopy class $w$.} This triangulation can then be used---as an input of existing FPT-algorithms, e.g., \cite{burton2017courcelle, burton2016parameterized, Burton2018, burton2013complexity, Maria17Beta1}---to compute difficult properties of $\manifold$ in running time singly exponential in $g$ and linear in the complexity of the presentation---for some reasonable definition of complexity. 
First, we introduce the additional background necessary for the statement and proof of \Cref{thm:app}.

\paragraph*{The mapping class group.} Recall the definition of a Heegaard splitting from \Cref{ssec:3mfds}. For the study of such splittings of a given genus $g$, it is crucial to get a grasp on isotopy classes of their attaching maps.
To this end, let $\altsurface_g^\star$ be the closed orientable genus $g$ surface $\altsurface_g$ with one {\em marked point} $\star \in \altsurface_g$ (for reasons provided later), and let $\operatorname{Homeo}^+(\altsurface_g^\star)$ denote the group of orientation-preserving homeomorphisms $f\colon\altsurface_g\rightarrow\altsurface_g$ with $f(\star) = \star$. The {\em mapping class group} $\MCG(\altsurface_g^\star)$ consists of the isotopy classes (also called {\em mapping classes}) of $\operatorname{Homeo}^+(\altsurface_g^\star)$, where isotopies are required to fix~$\star$.

The group $\MCG(\altsurface_g^\star)$ can be generated by some ``elementary'' homeomorphisms: Let $c \subset \altsurface_g$ be a {\em non-separating simple closed curve} (i.e., an embedding of the circle which does not split the surface into two connected components). Informally, a {\em Dehn twist along $c$} is a homeomorphism $\tau_c\colon\altsurface_g \rightarrow \altsurface_g$ where we first cut $\altsurface_g$ along $c$, twist one of the ends by $2\pi$, and then glue them back together \cite{dehn}. A commonly used---although non-minimal \cite{humphries1979generators}, cf.\ \cite{johnson1983structure,korkmaz2005generating}---set of Dehn twists to generate $\MCG(\altsurface_g^\star)$ is given by \Cref{thm:lickorish}.
For more background, we refer to \cite[Chapters 2 \& 4]{Benson12MCG}.

\begin{theorem}[Lickorish \cite{lickorish1962dehn}]
\label{thm:lickorish}
The group $\MCG(\altsurface_g^\star)$ is generated by the Dehn twists along the simple closed curves $\alpha_i, \beta_j$ $(1 \leq i,j \leq g)$ and $\gamma_k$ $(1 \leq k \leq g-1)$, as shown in \Cref{fig:lickorish}.
\end{theorem}

\begin{figure}[!ht]
	\centering
	\includegraphics[scale=0.405]{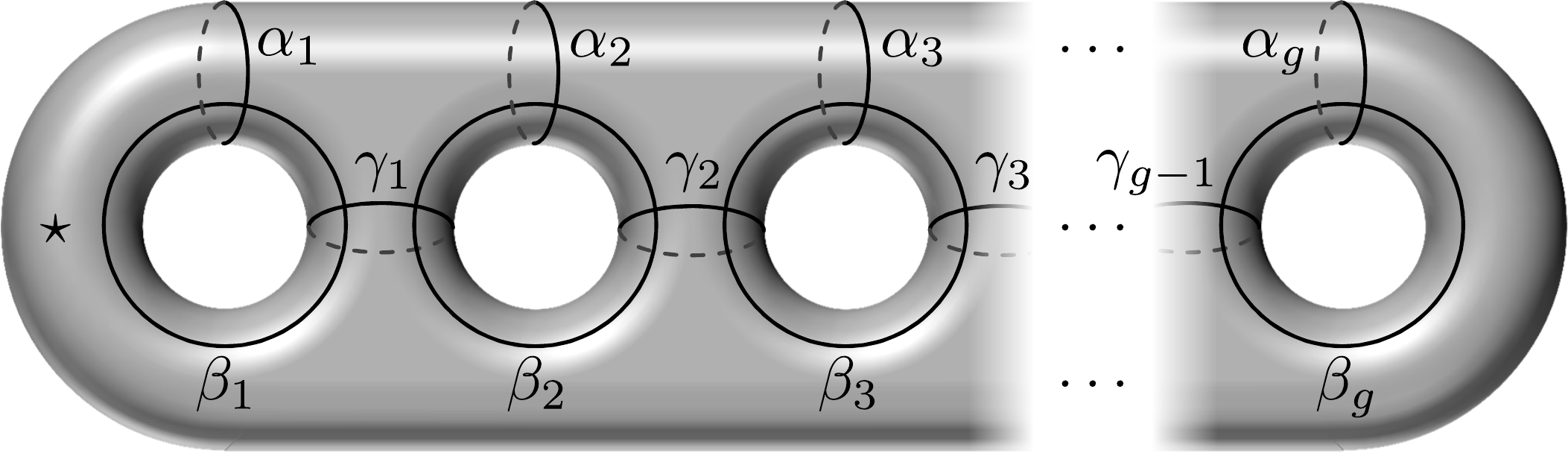}
	\caption{The marked surface $\altsurface_g^\star$ with $3g-1$ Dehn twists generating $\MCG(\altsurface_g^\star)$.}
	\label{fig:lickorish}
\end{figure}

\paragraph*{Flips.} Recall the definition of a layered triangulation from \Cref{ssec:layered}. Layered triangulations provide the following combinatorial view on the mapping class group. Let $\tri$ be a layered triangulation of a handlebody $\hbody$. By layering a tetrahedron $\Delta$ onto an edge $e \in \surface = \partial\tri$, we obtain another triangulation $\tri' = \tri \cup \Delta$ of $\hbody$, where, combinatorially, $\surface' = \partial\tri'$ is related to $\surface$ via an {\em (edge) flip} $e \leftrightarrow e'$, see  \Cref{fig:layering}.

\Cref{thm:layered3Mfd} implies that for any homeomorphism $f\colon\partial\hbody\rightarrow\partial\hbody$, there is a sequence $\tri^{(0)} \rightarrow \cdots \rightarrow \tri^{(i+1)} = \tri^{(i)} \cup \Delta \rightarrow \cdots \rightarrow \tri^{(m)}$ of layered triangulations of $\hbody$, descending to a {\em flip sequence} $\surface^{(0)} \rightarrow \cdots \rightarrow \surface^{(i)} = \partial\tri^{(i)} \rightarrow \cdots \rightarrow \surface^{(m)}$ of one-vertex triangulations of $\partial\hbody$, such that $f$, when considered as a map $\surface^{(0)}\rightarrow\surface^{(m)}$, is a simplicial isomorphism.

We can now state and prove the main theorem of this section.

\begin{theorem}
\label{thm:app}
Let $\gens$ be a set of Dehn twists generating $\MCG(\altsurface_g^\star)$, and $\manifold$ be a $3$-manifold given by a word $w \in \langle X \rangle$ (representing a mapping class). Then we can construct a layered triangulation of $\manifold$ of size $O(g + K|w|)$ and cutwidth $\leq 4g-2$ in time $O\left (K (|\gens| + |w|)\right )$, where $|w|$ denotes the length of $w$, and $K$ is a constant only depending on $g$ and $\gens$.
\label{thm:application}
\end{theorem}

\begin{proof}
Take a minimal layered triangulation $\tri'$ of the genus $g$ handlebody (e.g., the one shown in \Cref{fig:spine}). Then $\partial\tri'$ is a one-vertex triangulation of $\altsurface_g^\star$ (and its vertex is the ``marked point'' $\star$).
For every Dehn twist $\tau_c \in \gens$, let $k_c$ be the geometric intersection number of the edges of $\partial \tri'$ with the curve $c \subset \partial \tri'$ defining $\tau_c$. Computing $k_c$ is immediate due to the so-called {\em bigon criterion}, see \cite[Lemma 2.4.1]{Bell15Diss}. Moreover, according to \cite[Theorem 2.4.6]{Bell15Diss}, a flip sequence of length $O(k_c)$ realizing $\tau_c$ simplicially (cf.\ \Cref{ssec:layered} for definitions) can be found in $O(k_c)$ steps.  Let $K = \max\{k_c:\tau_c \in \gens\}$. It follows that we can compute the flip sequences for all generators in $\gens$ in time $O(K|\gens|)$.

With this setup it is straightforward to construct a layered triangulation $\tri$ of $\manifold $: Start with $\tri'$ and for every letter $\ell$ in $w$ (i.e., $\ell \in X$), perform at most $K$ layerings specified by the respective precomputed flip sequence; denote the resulting triangulation by $\tri''$. This can be done in $O\left(K|w|\right)$ steps altogether. $\tri$ is then completed by gluing a copy of $\tri'$ to~$\tri''$.

By construction, the number of tetrahedra in $\tri$ is at most $6g - 4 + K|w|$, and, following the proof of \Cref{thm:cutwidth}, $\cw{\dual(\tri)} \leq 4g-2$. The overall running time is $O\left (K (|\gens| + |w|)\right )$. 
\end{proof}

\begin{remark}
In this setting, it is natural to ask the following. Given a 3-manifold $\manifold$ of Heegaard genus $\mathfrak{g}(\manifold)=g$, what is the minimum length a word $w \in \MCG(\altsurface_g^\star)$ (with respect to some generating set $\gens$) that realizes $\manifold$? If we choose $\gens$ to be the set given by \Cref{thm:lickorish}, we obtain the so-called {\em Heegaard--Lickorish complexity} of $\manifold$ \cite{Cha19Compl3Mfds}.
\end{remark}

\section{Generating treewidth two triangulations using {\em Regina}}
\label{app:python}

The following Python functions can be executed using {\em Regina}'s Python interface or the shell environment within the {\em Regina} GUI. See \cite{burton2013regina, Regina} for more information.

First note that layered solid tori (of treewidth one) are readily available from within {\em Regina}. Given an $n$-tetrahedra triangulation \texttt{t}, we can insert a layered solid torus of type $\lst (a,b,a+b)$ at the end of \texttt{t} using the command \texttt{t.insertLayeredSolidTorus(a,b)}
with tetrahedra $\Delta_n , \ldots , \Delta_{n+k}$. Its boundary is always given by the triangles $\Delta_n (012)$ and $ \Delta_n (013)$ and the unique edge only contained in $\Delta_n$ is thus $(01)$.

Moreover, there is also the possibility of generating triangulations of the orientable Seifert fibered spaces over the sphere (with treewidth two), as well as layered lens spaces (with treewidth one), out-of-the-box using {\em Regina}.

\begin{verbatim}
# 3-punctured sphere x S^1
def prism1():
  t = Triangulation3()
  t.newSimplex()
  t.newSimplex()
  t.newSimplex()
  t.simplex(0).join(3,t.simplex(1),NPerm4())
  t.simplex(1).join(2,t.simplex(2),NPerm4())
  t.simplex(2).join(1,t.simplex(0),NPerm4(3,0,1,2))
  return t

# Moebius strip x~ S^1
def moebius():
  t = Triangulation3()
  t.newSimplex()
  t.newSimplex()
  t.newSimplex()
  t.simplex(0).join(0,t.simplex(1),NPerm4())
  t.simplex(1).join(1,t.simplex(2),NPerm4())
  t.simplex(0).join(1,t.simplex(1),NPerm4(0,2,3,1))
  t.simplex(1).join(3,t.simplex(2),NPerm4(2,0,1,3))
  t.simplex(0).join(2,t.simplex(2),NPerm4(1,3,0,2))
  return t

# Moebius strip union 3-punctured sphere x~ S^1
def ext_moebius():
  t = Triangulation3()
  t.insertTriangulation(moebius())
  t.insertTriangulation(prism1())
  t.simplex(0).join(3,t.simplex(5),NPerm4(2,0,1,3))
  t.simplex(2).join(2,t.simplex(3),NPerm4())
  return t

# Non-orientable genus g surface x~ S^1
def nonor_bundle(g):
  t = Triangulation3()
  for i in range(g):
    t.insertTriangulation(ext_moebius())
  for i in range(g-1):
    if i%2==0:
      t.simplex(6*i+4).join(0,t.simplex(6*i+10),NPerm4())
      t.simplex(6*i+5).join(0,t.simplex(6*i+11),NPerm4())
    if i%2==1:
      t.simplex(6*i+3).join(1,t.simplex(6*i+9),NPerm4())
      t.simplex(6*i+4).join(1,t.simplex(6*i+10),NPerm4())
  return t

# r-punctured sphere x S^1
def disk(r):
  if r < 0:
    return None
  if r == 0:
    t = prism1()
    t.insertLayeredSolidTorus(0,1)
    t.simplex(3).join(3,t.simplex(2),NPerm4(2,0,1,3))
    t.simplex(3).join(2,t.simplex(0),NPerm4(1,3,2,0))
    t.insertLayeredSolidTorus(0,1)
    t.simplex(6).join(3,t.simplex(0),NPerm4(3,2,0,1))
    t.simplex(6).join(2,t.simplex(1),NPerm4(0,3,1,2))
    t.insertLayeredSolidTorus(0,1)
    t.simplex(9).join(3,t.simplex(2),NPerm4(3,2,1,0))
    t.simplex(9).join(2,t.simplex(1),NPerm4(2,1,0,3))
    return t
  if r == 1:
    t = prism1()
    t.insertLayeredSolidTorus(0,1)
    t.simplex(3).join(3,t.simplex(2),NPerm4(2,0,1,3))
    t.simplex(3).join(2,t.simplex(0),NPerm4(1,3,2,0))
    t.insertLayeredSolidTorus(0,1)
    t.simplex(6).join(3,t.simplex(0),NPerm4(3,2,0,1))
    t.simplex(6).join(2,t.simplex(1),NPerm4(0,3,1,2))
    return t
  if r == 2:
    t = prism1()
    t.insertLayeredSolidTorus(0,1)
    t.simplex(3).join(3,t.simplex(2),NPerm4(2,0,1,3))
    t.simplex(3).join(2,t.simplex(0),NPerm4(1,3,2,0))
    return t
  if r >= 3:
    t = Triangulation3()
    for i in range(r-2):
      t.insertTriangulation(prism1())
    for i in range(r-3):
      if i%2 == 0:
        t.simplex(3*i+1).join(0,t.simplex(3*i+4),NPerm4())
        t.simplex(3*i+2).join(0,t.simplex(3*i+5),NPerm4())
      if i%2 == 1:
        t.simplex(3*i).join(1,t.simplex(3*i+3),NPerm4())
        t.simplex(3*i+1).join(1,t.simplex(3*i+4),NPerm4())
    return t

# Non-orientable genus g surface x~ S^1 with r punctures
def nonor_bundle_punct(g,r):
  t = Triangulation3()
  t.insertTriangulation(nonor_bundle(g))
  t.insertTriangulation(disk(r))
  t.simplex(3).join(1,t.simplex(6*g),NPerm4())
  t.simplex(4).join(1,t.simplex(6*g+1),NPerm4())
  return t

# Example
t = nonor_bundle_punct(6,12)
print t.detail()
\end{verbatim}

\end{document}